\documentclass[12pt]{article}
\usepackage{myArticle, myTools}

\usepackage{color,soul,mathrsfs}
\usepackage{bm} 
\usepackage{colortbl}      
\usepackage{enumitem}
\usepackage{booktabs}       
\usepackage{amsfonts}       
\usepackage{nicefrac}       
\usepackage{bbm}
\usepackage{algorithm}
\usepackage{algpseudocode}
\usepackage{tikz}
\usepackage{lipsum}
\usepackage{etoolbox}
\usepackage{microtype}      
\usepackage{hyperref}
\usepackage{url, mathabx}
\usepackage{csquotes}
\usepackage{subcaption}
\usepackage{pifont}
\usepackage{arydshln}
\usepackage{amsthm}
\usepackage{amssymb}

\title{Data-driven Policies For Two-stage Stochastic Linear Programs}
\author[1]{Chhavi Sharma\thanks{chhsh17@gmail.com}}
\author[1]{Harsha Gangammanavar\thanks{harsha@smu.edu}}

\affil[1]{Department of Operations Research and Engineering Management, Southern Methodist University, Dallas TX}

\def\cA{{\mathcal{A}}}
\def\cB{{\mathcal{B}}}
\def\cC{{\mathcal{C}}}
\def\cD{{\mathcal{D}}}

\def\cI{{\mathcal{I}}}

\def\cB{{\mathcal{B}}}

\newcommand{\rhsrv}{\tilde{b}}
\newcommand{\rhsobs}{b}
\newcommand{\bfalpha}{\boldsymbol{\alpha}}
\newcommand{\bfbeta}{\boldsymbol{\beta}}

\newcommand{\cmark}{\ding{51}} 
\newcommand{\xmark}{\ding{55}} 

\newif\ifpaper
\paperfalse

\begin{document}

\maketitle

\begin{abstract}
   A stochastic program typically involves several parameters, including deterministic first-stage parameters and stochastic second-stage elements that serve as input data. These programs are re-solved whenever any input parameter changes. However, in practical applications, quick decision-making is necessary, and solving a stochastic program from scratch for every change in input data can be computationally costly. This work addresses this challenge for two-stage stochastic linear programs (2-SLPs) with varying right-hand sides for the first-stage constraints. We construct a Piecewise Linear Difference-of-Convex (PLDC) policy by leveraging optimal bases from previous solves. This PLDC policy retains optimal solutions for previously encountered parameters and provides high-quality solutions for new right-hand-side vectors. Our proposed policy directly applies to the extensive form of the 2-SLP. When stage decomposition algorithms, such as the L-Shaped and Stochastic Decomposition, are applied to solve the 2-SLPs, we develop L-Shaped- and Stochastic-Decomposition-guided static procedures to train the policy. We also develop a sequential procedure that iteratively tracks the quality of the learned policy and incorporates new basis information to improve it. We assess the performance of our policy through analytical and numerical techniques. Our compelling experimental results show that the policy prescribes solutions that are feasible and optimal for a significant percentage of new instances.
\end{abstract}
\begin{history}
    Current version: \today
\end{history}

\section{Introduction}
In this work, we study the two-stage stochastic linear programming (2-SLP) problem of the following form
\begin{subequations}  \label{eq:2slp}
\begin{align} \label{eq:2slp-first-stage}
    \min~& f(x) \coloneqq c^\top x+ \expect{Q(x,\rv)}{\PP}\\
    \text{s.t.}~& \ x \in \set{X} \coloneqq \{x \in \RR^{d_x}~|~Ax = b, x \geq 0\}, \nonumber 
\end{align}
where $Q(x, \obs)$ is the value of a second-stage linear program, 
\begin{align} \label{eq:2slp-second-stage}
    Q(x,\obs) = \min_{y \in \mathbb{R}^{d_y}}~& q^\top y \\
    \text{s.t.}~& Wy = h(\obs)-T(\obs)x, \ y \geq 0. 
\end{align}
\end{subequations}
The decisions in 2-SLP are made in two stages. The decision $x$, known as a first-stage decision, is made before the random variable $\rv$ is realized. After implementing the decision $x$, a scenario of $\rv$ is observed, and the second-stage decision $y$ is determined. Without loss of generality, we assume that $A$ is of full row rank. The random vector represents the stochastic quantities that model the second stage, viz.,  $h \in \RR^{m_2}$, and $T \in \RR^{m_2 \times d_x}$. Throughout this article, we denote the random quantities using a tilde, such as $\rv$, and their observations without a tilde, as $\obs$. The 2-SLP has a diverse set of applications ranging from supply chain, power systems, telecommunications networks, pollution control, and chemical processes \citep{Wallace.Ziemba-2005}. The above problem is parametrized by deterministic quantities $A \in \RR^{m_1 \times d_x}, b \in \RR^{m_1}$, and $c \in \RR^{d_x}$, and the probability space associated with the random vector $\rv$ that we denote by $(\rvset, \Sigma_{\obs}, \PP)$. Typically, these parameters represent the available data to instantiate the 2-SLP problem. In this work, we study a setting in which we repeatedly need to solve 2-SLPs with parameters that may change over time. 

The field of stochastic programming (SP) has excellent methods to efficiently solve \eqref{eq:2slp} for given parameter values $(c,b, A,\PP)$, such as the L-Shaped \citep{Van.Wets-1969}, proximal bundle \citep{Ruszczynski-1986} and stochastic decomposition (SD) methods \citep{Higle.Sen-1994}. However, repeatedly solving a 2-SLP can be computationally expensive, even with efficient algorithms. Indeed, the most challenging component in solving 2-SLP is iteratively computing an approximation of the expectation-valued recourse function $\expect{Q(x,\rv)}{\PP}$. Even for small scenario sizes, solving the second-stage subproblems to obtain high-quality approximations of the recourse function and using these approximations to identify a here-and-now decision can be time-consuming. In many applications where optimization problems must be solved repeatedly and quickly, identifying optimal solutions from scratch under stringent deadlines is impractical. We identify one such example later in this section. Fortunately, these frequently solved decision problems share a common optimization structure, in which only some parameters are updated. This setting raises an interesting question: How can we leverage information from prior executions of SP algorithms to quickly identify implementable solutions for 2-SLP as its parameters change over time? This work addresses the question of when the \textit{right-hand sides of the first-stage constraints change}. Specifically, we model the first-stage right-hand-side vector as a random variable $\rhsrv$ associated with a probability measure $(\cB,\Sigma_\rhsobs,\mathbb{Q})$, where $\set{B} \subset \RR^{m_1}$ is a compact set. If the 2-SLP problem is feasible for almost all $\rhsobs \in \set{B}$, then given the information collected over the past $n$ executions of an SP algorithm to solve 2-SLP with possibly different observations of $\rhsrv$, we aim to identify a \emph{policy} that maps a right-hand-side vector to the optimal solution of the corresponding 2-SLP problem. 

Recall that every 2-SLP with finite support of second-stage uncertainty can be equivalently formulated as a linear programming (LP) problem, commonly referred to as the extensive scenario form in the SP literature \citep{ruszc03bk}, whose size scales with the number of second-stage scenarios.  For deterministic LP, \cite{walkup1969lifting} has formally demonstrated that the optimal solution is a continuous piecewise linear (PL) function of its right-hand side. In the ideal case, we can use this function to obtain an optimal solution whenever we encounter a new right-hand side. However, finding its closed-form expression requires identifying optimal basis matrices of the extensive-scenario LP for all possible right-hand sides. Such an endeavor is computationally expensive for large-scale stochastic programs that arise in real-world applications. Alternatively, we can solve the following optimization problem:
\begin{align} \label{eq:risk_neutral}
\min_{\phi \in \Phi_{\text{PL}}}~& \expect{f(\phi(\rhsrv))}{\rhsrv} \\
\text{s.t.}~& \ A\phi(\rhsrv) = \rhsrv, \phi(\rhsrv) \geq 0 \ \text{for almost every} \ \rhsrv, \notag
\end{align} 
where $\Phi_{\text{PL}}$ denotes the family of vector-valued continuous piecewise linear functions defined from $\RR^{m_1}$ to $\RR^{d_x}$. Problem \eqref{eq:risk_neutral} has at least one solution because the optimal solution of \eqref{eq:2slp} is a piecewise linear function of the right-hand side. A solution $\phi^\star$ of \eqref{eq:risk_neutral}, hereafter referred to as an optimal policy, can be used to get the optimal solutions of 2-SLP problems with varying right-hand sides. Clearly, both the objective and the constraints in \eqref{eq:risk_neutral} are nonconvex, rendering the problem intractable. Constraint-penalized approaches can be used to solve \eqref{eq:risk_neutral} by incorporating the constraints into the objective function via a suitable penalty term. Then, we can optimize the resulting penalized problem over the class of piecewise-linear functions using the algorithms developed, for instance, by \cite{Zhang.Liu.Zhao-2023}. However, penalization-based approaches are sensitive to the choice of the penalty parameter and may yield infeasible solutions. In either case, a further complication is an incomplete knowledge of the objective $f(\cdot)$ in 2-SLP, as it requires all the pieces of the polyhedral function $\expect{Q(x,\rv)}{\PP}$.

The 2-SLP problem has nice structural properties regarding the objective function and the relationship between the optimal basis and the right-hand sides. These properties can inform the number of pieces in the PL policy and allow us to leverage information from prior executions of a solution algorithm. Thus, we will bypass the intractable nonconvex problem \eqref{eq:risk_neutral} to design a PL policy that efficiently delivers high-quality solutions whenever a new right-hand side is observed.

\subsection{Relations to Parametric Optimization and Decision Rules}
The 2-SLP can be viewed as a parametric stochastic program, which under finite support, is equivalent to a parametric LP. Sensitivity analysis is a fundamental approach for post-optimality analysis of the current basis of a parametrized LP problem \citep{bertsimas1997introduction}. Suppose a solver produces an optimal basis $\bfB$ of an LP problem at right-hand side $b$. Sensitivity analysis tells us that $\bfB$ remains optimal within a range of perturbations on $b$. Beyond this sensitivity range, a new basis yields an optimal solution. This tool is useful for understanding how changes to parameters affect the optimal basis. Our policy design builds upon this understanding to address the issue of repeatedly solving large-scale, complex stochastic programs. 

The 2-SLP problems are $\#$P-hard \citep{Dyer.Stougie-2006} and intractable even when seeking medium-accuracy solutions \citep{Shapiro.Nemirovski-2005}.  Several previous works have addressed parametric SP \citep{Robinson.Wets-1987, Dupavcova-1990, Royset.Chen.Louis-2025}. However, they stem from the stability analysis of optimal values and solutions under perturbations in the underlying distribution of second-stage uncertainty. Moreover, these works are analytical in nature, and their adoption into a solution algorithm, while appropriate, has not yet been addressed to the best of our knowledge. Another related stream of work in SP concerns the decision rules \citep{Kuhn.Wiesemann.Georghiou-2011, Bertsimas.Goyal-2012, Georghiou.Wiesemann.Kuhn-2015, Bertsimas.Bidkhori-2015, Georghiou.Kuhn.Wiesemann-2019}, where the aim is to address the computational difficulty of solving SP by assuming that the second-stage decision is either a linear or a nonlinear function of uncertainty scenarios $\obs$. The SP problem is then solved by replacing the second-stage decision with an appropriate mapping. The setting we study in this paper is distinct from that addressed in previous works on decision rules. We aim to find decision rules (which we call policies) for 2-SLPs that help repeatedly solve perturbed SP problems. We do not impose any restrictions on the behavior of second-stage decisions in relation to the observed scenario. Moreover, we concentrate on the propagation of information across different executions of the 2-SLP algorithms. 

A relevant study for our setting is the recent work \citep{royset2023risk} that designs decision rules for mixed-binary parametric optimization problems. This work lays a solid theoretical foundation for studying the asymptotic behavior of decision rules derived from training problems. However, two major difficulties prevent us from using their decision rules. The first is that it lacks guidance on how to solve the sequence of nonconvex training problems of the form \eqref{eq:risk_neutral}, which is required to adopt their approach in our setting. The second bottleneck arises because their feasible decision rules are prescribed for problems where the constraints are independent of the parameters, whereas we deal with varying constraint parameters in our 2-SLP setting.

\subsection{Contributions}
In this paper, we design data-driven Piecewise Linear Difference-of-Convex (PLDC) policies to solve 2-SLP problems with adaptive right-hand sides of the first-stage constraints. These policies also apply to deterministic LP problems that must be solved repeatedly, with right-hand sides varying across instantiations. Since our goal is to devise a policy that delivers near-optimal or at least feasible solutions to 2-SLP problems, which are repeatedly solved in practice, the process of finding such a policy should be simpler than that of solving the 2-SLP problems from scratch. In this regard, we summarize the main contributions of this paper below:

\begin{enumerate}
    \item We utilize the properties of the optimal solution map and difference-of-convex functions to design a linear training problem. A solution to this training problem generates the desired PLDC policy for deterministic LP problems. The training problem is constructed using data comprising right-hand sides and the optimal solution and optimal basic variables of the corresponding LP. We analytically show that our proposed policy yields an optimal solution whenever the right-hand-side vector lies in the convex hull of the right-hand sides in the training data. This part of the contributions is presented in section \ref{sec:pldc_lp}.
    \item In section \ref{sec:pldc_2slp}, we adopt the PLDC policy of deterministic LP to 2-SLP problems solved using the L-Shaped or SD methods as solution algorithms. We devise a static procedure that constructs a consolidated master problem by reusing the local polyhedral approximations and terminal solutions generated by executions of the decomposition algorithm to solve a fixed number of 2-SLPs with different right-hand sides. Based on this consolidated problem, we formulate a training problem whose solution yields the parameters of the PLDC policy for 2-SLPs. It is also worth noting that the resulting PLDC policy-training problem can be solved using off-the-shelf LP solvers.
    \item Next, in section \ref{sec:seqproc}, we develop a unified decomposition algorithm-guided sequential procedure using either L-Shaped or SD methods. This sequential procedure iteratively feeds information from new solves to train the policy iteratively and tracks its quality. We maintain a training dataset that includes the right-hand sides, optimal solutions, optimal values, and local approximations. Adding new information to the training data increases the size of the training problem, which we control by appending it to the training set from 2-SLPs that are infeasible or suboptimal under the current policy. We prove that the sequential procedure returns a viable policy in a finite number of rounds under reasonable assumptions.
    \item We test our PLDC policies on multiple 2-SLP problems, ranging from small to large scale. This computational study demonstrates that our proposed policies retain the optimal solutions from previous solves and produce optimal solutions for most new instances. Experiments with the sequential procedure show that it arms the training process to efficiently incorporate new information, thereby improving the quality of the policy-prescribed solutions in an online manner.
\end{enumerate}
The main goal of designing policies that prescribe optimal or near-optimal solutions to large-scale SP problems within a reasonable time is to enable the practice of operations research tools, particularly SP models and algorithms, in real-world settings. We identify one such application area below.

\subsection{A Motivating Example} \label{app:motivatingexample}
Economic dispatch (ED) is a key problem that power system operators solve to plan electricity generation from various sources, including hydro, thermal (e.g., coal or gas), and renewable (e.g., wind and solar). The goal is to determine the amount of electricity each generating unit should produce to minimize total operating cost. This planning problem comprises several constraints, including the system flow-balance equation; physical constraints on generators (capacity, ramping, etc.) and transmission lines (capacity, power flows, etc.); and water-conservation equations for the reservoirs. The stochastic variant of the ED problem can be cast as a 2-SLP (see \cite{gangammanavar2015stochastic}), in which the first-stage decision is to determine the generation levels of slow-responding generators fueled, for instance, by coal and nuclear. Once the actual demand is revealed, the second stage determines the generation levels of fast-responding gas generators needed to align total generation with the realized demand and determine the power flows through the transmission network. This 2-SLP problem includes several parameters, such as electricity price, resource availability, and weather-related factors, which are integral to the model. In practice, the ED problem is solved several times a day. For instance, the California Independent System Operator must solve the ED problem every 5 minutes \cite{caiso2022report}. At this timescale, the parameters change continuously due to market fluctuations and varying weather conditions. Thus, it becomes necessary to re-solve a 2-SLP problem whenever any of these parameters change, making the stochastic ED an excellent example of the decision setting we study in this paper. Similar decision settings also arise in applications involving production planning, scheduling, resource allocation, and inventory control. The current state-of-the-art algorithms are incapable of delivering high-quality solutions in the desired timeframe. For instance, the SD algorithm takes about 20 minutes to deliver a statistically optimal solution to an ED instance for a medium-sized power network, as reported in \cite{gangammanavar2015stochastic}. This experience motivates us to avoid solving these SP problems from scratch and instead explore policy designs that are easy to train and execute within the tight time bounds.

\subsection{Notations and Background}
Throughout this paper, we assume that the scalars, sets, vectors, and matrices in a collection are denoted by the indices in their superscripts. For example, vector $b^i$ denotes the $i$-th right-hand-side vector, $B^i$ denotes $i$-th matrix and $I^i$ denotes $i$-th set of indices $\{i_1,i_2,\ldots \} \subseteq \{1,2,\ldots,d_x\}$. We use subscripts to denote components of a vector, e.g., $b^i_k$ denotes the $k$-th component of $b^i$. The boldface symbols $\mathbf{0}$ and $\mathbf{1}$ denote, respectively, vectors of zeros and ones of appropriate dimensions, as will be clear from the context. Let $\mathbf{0}_n$ denote a square matrix of size $n$ with all zero entries. An optimal solution and the corresponding optimal value of \eqref{eq:2slp} at a right-hand side $b$ are denoted by $x^\star(b)$ and $v^\star(b)$ respectively. For any positive integer $a$, $[a]$ denotes a set of indices $\{1,2,\ldots,a\}$. We denote the convex hull of a finite set $\set{S} = \{s_1,s_2,\ldots,s_n \}$ by $\text{Conv}(\set{S}) = \{\lambda_1s_1+\lambda_2s_2+\ldots+\lambda_ns_n \ \vert \ \sum_{i=1}^n \lambda_i = 1, \lambda_i \geq 0 \}$. A continuous function $f:\RR^{m} \rightarrow \RR^{d}$ is piecewise linear if there exists a finite set of linear functions $f_i(x) = D^ix$, $i= 1,\ldots, K$ such that $f(x) \in \{f_1(x),\ldots,f_K(x)\}$ for all $x \in \RR^m$ \citep{scholtes2012introduction}. A function $f:\RR^{m} \rightarrow \RR^{d}$ is called a difference-of-convex function if there exists convex functions $g_1:\RR^{m} \rightarrow \RR^{d}$ and $g_2:\RR^{m} \rightarrow \RR^{d}$ such that $f(x) = g_1(x) - g_2(x)$ for all $x \in \RR^{m}$, where difference is defined componentwise \citep{Le.Hoai.Pham-2018}.

\section{Policy Design} \label{sec:pldc_lp}
In this section, we construct a policy for deterministic LP problems with adaptive right-hand sides. We focus on a deterministic LP to communicate the key ideas involved in policy design and construction to the reader without complicating the notation and adapt the steps to 2-SLP in the subsequent sections of the paper. To this effect, we consider the standard form of a deterministic LP problem: $ \min~\{c^\top x~|~x \in \set{X}\}$. 
Notice that we retain the notations and form of the first-stage program in \eqref{eq:2slp-first-stage}. The optimal solutions of LP problems are extreme points of the feasible region, also called basic feasible solutions. We can partition the solution vector into basic variables, indexed by $B$ and denoted by $x_B$, and nonbasic variables, indexed by $N$ and denoted by $x_N$. The columns of the constraint matrix $A$ corresponding to basic variables constitute the basis that we denote by $\bfB$. If $x$ is an optimal solution, then we term the corresponding matrix $\bfB$ as the optimal basis. Using the basis matrix, the basic feasible solution is given by $x_B = \bfB^{-1} b, \  x_N = 0$.

When some or all entries of the right-hand-side vector $b$ are changed, say to a new vector $b^\prime$, the perturbation theory of LP \citep{bertsimas1997introduction} says that the current basis $\bfB$ is optimal to $b^\prime$. Typically, the post-optimality sensitivity analysis reveals the sensitivity range when each element is perturbed in turn. In particular, the basis $\bfB$ remains optimal as long as $\bfB^{-1}b^\prime$ is nonnegative. Beyond a certain level of perturbation, a new basis matrix, say $\bfB^\prime$, yields the optimal solution. Hence, the optimal solution of an LP problem is a continuous, piecewise-linear function of the right-hand-side vector. In other words, $\phi^*$ is a continuous, piecewise linear function. This result is well established in \cite{walkup1969lifting} for parametric LPs.

More formally, following the Basis Decomposition Theorem of \cite{walkup1969lifting}, we can decompose the closed convex polyhedral cone $\text{pos}(A) \ = \ \{b \ \vert \ b = Ax, x \geq 0 \}$ into a finite number, say $\bar{L}$, of closed convex polyhedra, which are referred to as cells. The $m_1$ columns of $A$ corresponding to the edges of an $m_1$-dimensional cell $\cC$ form an optimal basis for all $b \in \cC$. If $\bfB(\cC)$ is the optimal basis corresponding to cell $\cC$, then a linear piece given by $(\bfB(\cC))^{-1}b$, which we term the optimal piece, recovers the optimal value of basic variables for all $\rhsobs \in \cC$. Moreover, since the $\bar{L}$ cells cover the entire $\text{pos}(A)$, the corresponding bases, along with the trivial zero function for non-basic variables, provide the means to construct the optimal piecewise linear policy $\phi^*$. Since the number of cells is finite, the policy $\phi^*$ has a finite number of pieces.

\subsection{Data-driven Policy Approximation}
Obtaining a closed-form expression of the optimal piecewise policy $\phi^*$ requires discovering all possible bases, a number that scales exponentially in the problem dimension. This step is computationally prohibitive, particularly for large-scale LP models commonly encountered in real-world applications and in SP problems that we encounter in later sections. Alternatively, we can adopt a data-driven approach that uses the optimal bases obtained by solving LP problems for a finite set of right-hand-side parameters, say $n$. Let $\widehat{\set{B}} \coloneqq \{b^j\}_{j \in [n]}$ denote the finite collection of right-hand sides and $\widehat{\set{X}} \coloneqq \{x^\star(b^j)\}_{j \in [n]}$ the optimal solution vectors at these right-hand sides. Furthermore, let $\{\bfB^\ell\}_{\ell=1}^L$ denote the collection of optimal bases that we discover (notice $ L \leq n$ and $L \leq \bar{L}$). We construct approximate cells by placing the right-hand sides $\widehat{\set{B}}$ that have the same optimal basis $\bfB^\ell$ into a cell $\cC^\ell$, which we define as $\cC^\ell = \{ b \in \RR^{m_1} \ \vert \ x_B^\star(b) = (\bfB^\ell)^{-1}b  \} \subseteq \widehat{\set{B}}$. Figure \ref{fig:celldemo} illustrates such cells for an LP problem with two equality constraints and five variables. The right-hand-side vectors belonging to the same cell are plotted using the same marker (and color).

Using the above elements, we construct a data-driven approximate policy $\hat{\phi}: \set{B} \rightarrow \RR^{d_x}$ such that for any $b \in \set{B}$ as
\begin{align}
    \big( \hat{\phi}(b) \big )_k & = \begin{cases} \big((\bfB^{\ell^*})^{-1}b \big)_k & \text{if } k \in B^{\ell^*}\\ 0 & \text{if } k \in N^{\ell^*}, \end{cases} \label{eq:optimalpolicy}
\end{align}
where $\ell^* \in \{ \ell \ \vert \ (\bfB^{\ell})^{-1}b \geq 0, \ \ell \in [L] \}$. Notice that the basis $\bfB^{\ell^*}$ is always optimal since the cost coefficients are unchanged. It is also worthwhile to note that $\bfB^{\ell^*}$ remains optimal for any $\rhsobs \in \text{Conv}(\cC^{\ell^\star})$, even if $\rhsobs \notin \widehat{\set{B}}$. On the contrary, since we may have discovered only a subset of optimal bases in our data-driven setting, we may encounter a case where no such $\ell^*$ is found. As a consequence, all bases in the set $\{\bfB^\ell\}_{\ell=1}^L$ are infeasible to a right-hand side $\rhsobs$. Therefore, $\hat{\phi}$ is only an approximation of the true policy $\phi^\star$. 
\begin{figure}[!tb]
    \centering
    \includegraphics[width=0.5\linewidth]{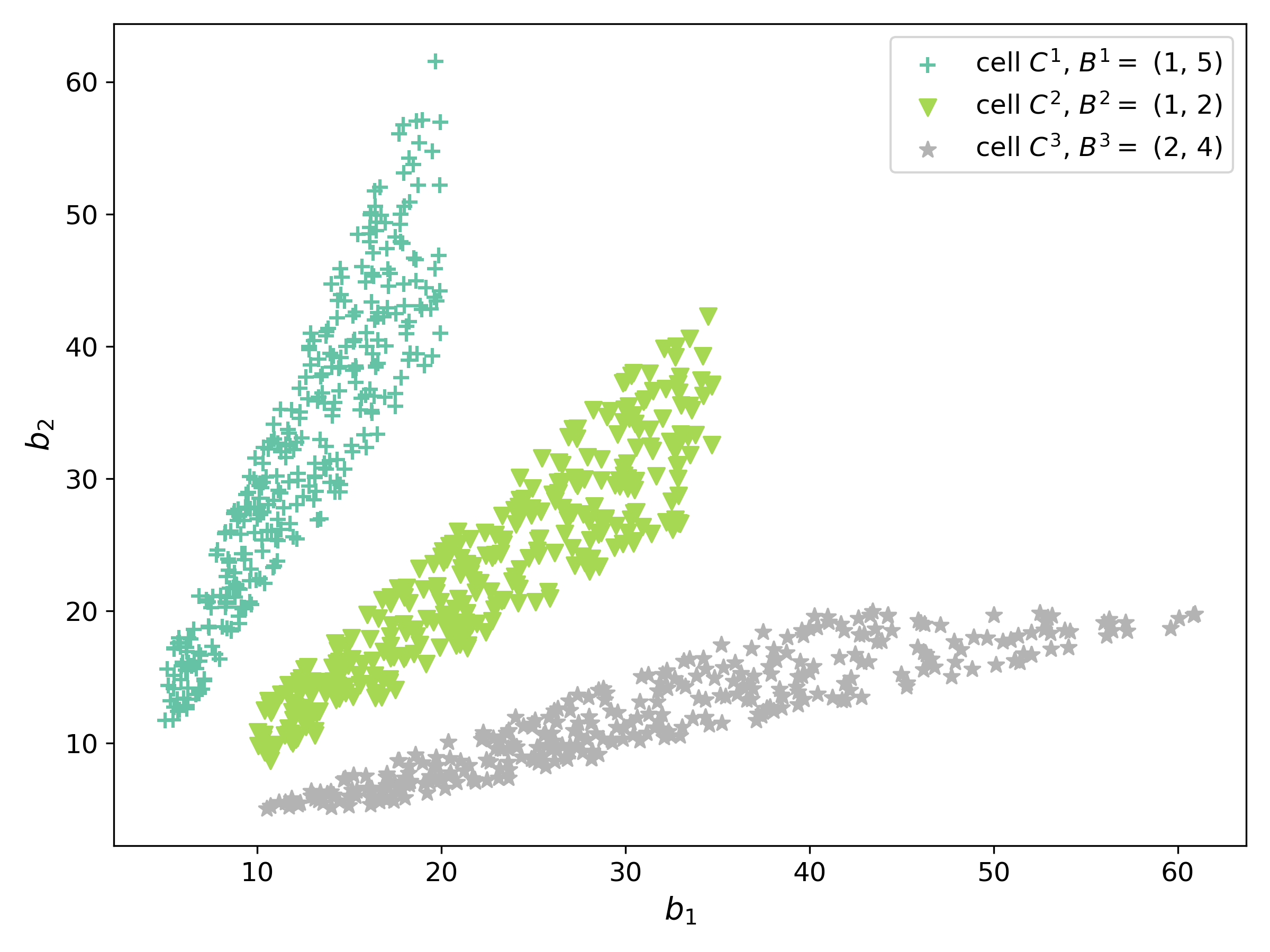}
    \caption{ \centering (Color online) Approximate cells for an LP problem with two constraints and five variables. $b_1$ and $b_2$ denote the right-hand-side values of the first and second constraints, respectively.}
    \label{fig:celldemo}
\end{figure}

While $\hat{\phi}$ is a viable policy, it requires considerable memory to store the dense inverses of the bases. Alternatively, one can store the column indices of the coefficient matrix $A$, and, upon observing a new $\rhsobs$, compute the matrix inverse. Clearly, performing inverse computations and conducting feasibility checks for every new $\rhsobs$ also entails high costs. In what follows, we leverage the piecewise-linear relationship between the optimal solutions and the right-hand sides of the constraints to construct an alternative characterization of the policy $\hat{\phi}$. This characterization yields a functional form that does not require feasibility checks or the storage or computation of the bases' inverses.

To establish such a characterization, we exploit the fact that we can express a piecewise-linear function as a difference-of-convex function \citep{kripfganz1987piecewise,scholtes2012introduction}. That is, we can represent a piecewise linear vector-valued function $\phi(b)$ as $\phi(b) = \phi_1(b) - \phi_2(b)$ where $\phi_1:\RR^{m_1} \rightarrow \RR^{d_x}$ and $\phi_2:\RR^{m_1} \rightarrow \RR^{d_x}$ are convex piecewise-linear functions componentwise. In our context, we aim to discover the convex functions $\phi_1$ and $\phi_2$ that we can use to reconstruct the optimal policy $\phi^*$ or at least its data-driven approximation $\hat{\phi}$. We can deem such a reconstruction as a success if, when restricted to $\text{Conv}(\cC^\ell)$, $\phi(b)$ recovers the piece $(\bfB^\ell)^{-1}b$ componentwise for basic variables and takes zero value for nonbasic variables. In this case, $(\phi(b))_k = (\phi_1(b))_k - (\phi_2(b))_k = ((\bfB^\ell)^{-1}b)_k $ for $b \in \text{Conv}(\cC^\ell)$ and $k \in B^\ell$. A key observation is that $\phi_1$ and $\phi_2$ can each have multiple pieces and still yield a linear piece, since the difference between two piecewise linear functions can be linear. That is, a cell may have more than one piece in $\phi_1$ and $\phi_2$. However, we show in Lemma \ref{lem:single_subgrad_existence} below that there exist \textit{convex} functions $\phi_1$ and $\phi_2$ such that, restricting them to $\text{Conv}(\cC^\ell)$, they are linear functions. We prove such existence in terms of one subgradient for all right-hand sides belonging to a given cell $\cC^\ell$. 
\begin{lemma} \label{lem:single_subgrad_existence}  Let $\phi(b) = \phi_1(b) - \phi_2(b)$ be a continuous piecewise-linear function, and let the right-hand side set $\widehat{\set{B}}$ be partitioned into $L$ cells. For every cell $\cC^\ell$, there exists $u^{\ell} = [u^\ell_1;\ldots;u_{d_x}^\ell] \in \RR^{d_x \times m_1}$, $v^{\ell} = [v^\ell_1;\ldots;v^\ell_{d_x}] \in \RR^{d_x \times m_1}$ such that $u^\ell_k \in \partial (\phi_1(b))_k$ and $v^\ell_k \in \partial (\phi_2(b))_k$ for all $b \in \text{Conv}(\cC^\ell)$ and $\ell \in [L]$.
\end{lemma}
Our approach to proving this lemma is to fix a cell and show that every linear function can be represented as the difference of two linear functions. Moreover, we show that the collection of linear functions identified across all the cells forms piecewise-affine convex functions $\phi_1$ and $\phi_2$. We provide a complete proof of Lemma \ref{lem:single_subgrad_existence} in Appendix \ref{appendix:lemmaproof}.

From Lemma \ref{lem:single_subgrad_existence} and selecting a pair $(b^\ell, x^\star(b^\ell))$ from each cell $\cC^\ell$, we obtain the following characterization of the policy:
\begin{align}
(\phi(\rhsobs))_k & = \max_{\ell \in [L]} \Big((u^\ell_k)^\top (\rhsobs -b^\ell) + (x^\star(\rhsobs^\ell))_{k} + z^\ell_k  \Big)\nonumber \\ 
& \qquad - \max_{\ell \in [L]} \Big( (v^{\ell}_k)^\top (\rhsobs -b^\ell)+ z^\ell_k \Big).\label{eq:pldc_policy}
\end{align}
Here, the parameter vector $\{\theta_k \}_{k=1}^{d_x} = \{(u^{\ell(j)}_k,v^{\ell(j)}_k,z_k^j)_{j=1}^n \}_{k=1}^{d_x} \in \RR^{(2m_1L+n)d_x}$ is a solution that satisfies the following system of inequalities:
\begin{subequations} \label{eq:policyinequality}
\begin{align}
    & \inner{u^{\ell(j)}_k, b^i-b^j} + (x^\star(b^j))_k + z^j_k \nonumber \\
    & \qquad \leq  (x^\star(b^i))_k + z_k^i \ \text{for all} \ i,j \in [n],  \label{eq:phi_diff_u}  \\
    & \inner{v^{\ell(j)}_k, b^i-b^j} + z^j_k \leq z_k^i  \ \text{for all} \ i,j \in [n].  \label{eq:vz}
\end{align}
\end{subequations}
We call $\phi(b)$ a \textit{data-driven} piecewise linear difference-of-convex (PLDC) policy as it is obtained using a finite number of data points $\{b^j,x^\star(b^j) \}_{j=1}^n$. Due to the PL nature of the optimal solution map and using Lemma \ref{lem:single_subgrad_existence}, there exists at least one solution to the inequalities in \eqref{eq:policyinequality}, which we establish in Theorem \ref{thm:pldc_properties} along with policy properties.
\begin{theorem} \label{thm:pldc_properties} For a LP in the standard form, let $\widehat{\set{B}}$ and $\widehat{\set{X}}$ be given. Then, the inequalities in \eqref{eq:policyinequality} have at least one feasible solution $\{\theta_k \}_{k=1}^{d_x} = \{(u^{\ell(j)}_k,v^{\ell(j)}_k,z_k^j)_{j=1}^n \}_{k=1}^{d_x}$ and the policy $\phi(b)$ in \eqref{eq:pldc_policy} constructed from $\{\theta_k \}_{k=1}^{d_x}$ satisfies the following:
\begin{align}
    \phi(b) = x^\star(b) \ \text{for all} \ b \in \text{Conv}(\cC^\ell), \ \ell \ \in \ [L]. \label{eq:policyproperties}
\end{align}
\end{theorem}
\proof{Proof.} We divide the proof into two parts. In the first part, we show the existence of a feasible solution. In the second part, we derive the property \eqref{eq:policyproperties} of policy $\phi(b)$.

\textbf{Feasible solution existence:} We can write optimal solution map $\phi^\star(b)$ as $ \phi^\star(b) = \phi^\star_1(b) - \phi^\star_2(b)$. 
Let $(\phi^\star_2(b^i))_k = z^i_k$ and $(\phi^\star_2(b^j))_k = z^j_k$ and consider
\begin{align}
(\phi^\star(b^i))_k  &= (\phi^\star_1(b^i))_k - (\phi^\star_2(b^i))_k \nonumber \\
& \geq (\phi^\star_1(b^j))_k + \inner{u^{\ell(j)}_k, b^i-b^j} - (\phi^\star_2(b^i))_k \nonumber \\
& = (\phi^\star(b^j))_k + \inner{u^{\ell(j)}_k, b^i-b^j} + z^j_k - z^i_k. \label{eq:diff-xkij}
\end{align}
The first inequality follows from the definition of subgradient. We recall that there exists a piecewise linear optimal solution map $\phi^\star$ such that $\phi^\star(b) = x^\star(b)$ for all $b \in \cB$. Using this, we rewrite \eqref{eq:diff-xkij} as
\begin{align}
& (x^\star(b^i))_k - (x^\star(b^j))_k \nonumber \\
& \geq \inner{u^{\ell(j)}_k, b^i-b^j} - z_k^i + z^j_k \ \text{for all} \ i,j \in [n]. \label{eq:xstar_diff_u}
\end{align}
For the $\phi^\star_2$ function, we have
\begin{align}
z_k^i & \geq z^j_k + \inner{v^{\ell(j)}_k, b^i-b^j}. \label{eq:phi2subgrad}
\end{align}
From Lemma \ref{lem:single_subgrad_existence}, there exists $u^{\ell(j)}_k$ and $v^{\ell(j)}_k$ satisfying \eqref{eq:diff-xkij} and \eqref{eq:phi2subgrad}; completing the proof of the first part.

\textbf{Policy properties:} Consider a cell $\cC^\ell$. Then following the arguments of the proof of Lemma \ref{lem:single_subgrad_existence}, $\phi_1$ and $\phi_2$ restricted to $\text{Conv}(\cC^\ell)$ are linear functions. It implies that inequalities in \eqref{eq:policyinequality} are tight for all $b^i, b^j \in \cC^\ell$ i.e.,
\begin{align*}
\inner{u^{\ell}_k, b^i-b^j} + x^\star(b^j)_k + z^j_k & = x^\star(b^i)_k + z_k^i \\
 \inner{v^{\ell}_k, b^i-b^j} + z^j_k  & = z_k^i. \label{eq:vl_tight}
\end{align*}
For any other $\tilde{\ell} \neq \ell$, we only have the inequalities in \eqref{eq:policyinequality}. This implies that $\max_{\ell' \in [L]}(u^{\ell'}_k)^\top (\rhsobs^i -b^{\ell'}) + x^\star(\rhsobs^{\ell'})_{k} + z^{\ell'}_k =  x^\star(b^i)_k + z_k^i$ and $\max_{\ell' \in [L]} (v^{\ell'}_k)^\top (\rhsobs^i -b^{\ell'})+ z^{\ell'}_k = z^i_k $ and hence $\phi(b^i) = x^\star(b^i)$. Consider a point $\hat{b}^\ell = \sum_{i =1}^{\vert \cC^\ell \vert} \lambda^{i} b^{n_i} \in \text{Conv}(\cC^\ell)$. Then we can clearly see that $(\phi_1(\hat{b}^\ell))_k = \sum_{i=1}^{\vert \cC^\ell \vert} \lambda^i (x^\star(b^{n_i})_k + z^{n_i}_k)$ and $(\phi_2(\hat{b}^\ell))_k = \sum_{i=1}^{\vert \cC^\ell \vert} \lambda^i z^{n_i}_k$ and hence $\phi(\hat{b}^\ell) = x^\star(\hat{b}^\ell)$, completing the proof. \hfill \ifpaper \Halmos \else \qed \fi

The striking advantage of PLDC policy \eqref{eq:pldc_policy} is that it can be obtained using off-the-shelf solvers as it requires a solution feasible to constraints in \eqref{eq:policyinequality}. This efficiency can be further enhanced by observing that many entries in $u^\ell$, $v^\ell$, and $z^\ell$ may end up being zero when the number of nonbasic variables is large, and a few components of the right-hand-side vector are varying. Thus, we effectively restrict the solutions of \eqref{eq:phi_diff_u}-\eqref{eq:vz} to the ones having minimum $\ell_1$ norm. This step yields a convex optimization problem, which we refer to as the training problem.
\begin{subequations} \label{eq:training_problem}
\begin{align}
\min_{u,v,z}~&  \sum_{l=1}^{L} \sum_{k =1}^{d_x} \Vert u^\ell_k\Vert_1 + \Vert v^\ell_k \Vert_1 + \sum_{i=1}^n \Vert z^i\Vert_1  \\
\text{s.t.}~& \ \inner{u^{\ell(j)}_k, b^i-b^j} + x^\star(b^j)_k + z^j_k  \nonumber \\
& \leq  x^\star(b^i)_k + z^i_k , \forall  \ i,j \in [n], \ k \in [d_x], \\
& \inner{v^{\ell(j)}_k, b^i-b^j}  + z^j_k   \leq z^i_k , \forall \ i,j \in [n], \ k \in [d_x].
\end{align}
\end{subequations}
With the solution of this training problem, we can construct the PLDC policy \eqref{eq:pldc_policy}. The resulting PLDC policy has several beneficial properties, which we identify below. 
\begin{enumerate}
    \item The policy recovers an optimal piece $(\mathbf{B}^\ell)^{-1}b $ for all $b$ within the convex hull $\text{Conv}(\mathcal{C}^\ell)$.  As the optimal solutions for past right-hand-side values are obtained after a significant investment of resources (viz., executing an SP algorithm), the ability to recover these solutions exactly is a minimal expectation that our policy design certainly meets. Note that an optimal piece $(\mathbf{B}^\ell)^{-1}b $ over an entire cell may not be recovered if the number of data points that belong to the cell is limited.
    \item Choosing one subgradient per cell rather than a subgradient for every observation of the right-hand-side vector is one of the crucial features of our characterization, which alleviates the issue of holding a large number of pieces. \cite{siahkamari2020piecewise} propose a policy that shares certain similarities with our PLDC policy \eqref{eq:pldc_policy}, albeit for piecewise linear regression. In their design, they need one piece per data point. However, such an approach does not work in our problem setting and fails to recover the pieces $(\bfB^\ell)^{-1}b$. We elaborate on these differences in the behavior in Appendix \ref{appendix:learningPL}.
    \item It is also worth recalling that the methodologies in \cite{royset2023risk} and \cite{Zhang.Liu.Zhao-2023} solve computationally expensive nonconvex problems sequentially to learn a piecewise linear function. In contrast to their approach, our training problem \eqref{eq:training_problem} is a convex optimization problem that can be reformulated into an LP and solved using off-the-shelf LP solvers. Moreover, our policy characterizes the optimal bases from previous solves in a functional form and prescribes the solution for a new right-hand side. Therefore, executing the policy to identify a solution to a new instance requires minimal computational effort and can be delivered within tight time bounds. 
\end{enumerate}
While the policy design procedure we present in this paper pertains to perturbations to the constraint right-hand sides, it also applies to the relatively easier setting of perturbations to the cost coefficients. In the next section, we adapt the training problem \eqref{eq:training_problem}, and subsequently, the PLDC policy \eqref{eq:pldc_policy} to 2-SLPs with varying right-hand sides in the first stage.

\section{PLDC Policy for Two-Stage Stochastic Linear Programming} \label{sec:pldc_2slp}
In this section, we construct data-driven policies for 2-SLPs using the tools established in the previous section. Before presenting the policy design procedures, we state the assumptions for the 2-SLP models with a right-hand side belonging to $\cB$. 
\begin{enumerate}[label=A\arabic*.]
\item The solution set $\set{X}$ of first-stage decisions is compact and set of scenarios $\rvset$ is finite.
\item The relatively complete recourse property is satisfied. \label{assum:completeRecourse}
\item The recourse function $Q(x,\obs) > -\infty$ almost surely. \label{assum:boundedSubproblem}
\item The recourse matrix ($W$) is deterministic, i.e., the fixed recourse property is satisfied. \label{fixedRecourse} 
\end{enumerate}
The assumption \ref{assum:completeRecourse} implies that the second-stage program is feasible for all $x \in \set{X}$ and $\obs \in \rvset$, implying $Q(\cdot, \obs) < \infty$, almost surely. With \ref{assum:boundedSubproblem}, the dual feasible set of the second-stage is nonempty, almost surely. 

The PLDC policy \eqref{eq:pldc_policy} directly applies to the extensive scenario form of a 2-SLP problem with finite support. Recall that the policy training dataset includes optimal solutions and the indices of their basic variables. Obtaining these requires solving extensive scenario forms, which becomes computationally intensive as the number of second-stage uncertainty scenarios increases, rendering the direct application of the policy design procedure in \S\ref{sec:pldc_lp} computationally prohibitive. It is also well-known that the expected recourse function and consequently the first-stage objective $f(x)$ are polyhedral for a 2-SLP with finite support \citep{shapiro2021lectures}. Using this result, a 2-SLP is written in the decomposed form as
\begin{align}
    \min_{x \in \set{X},\eta}~& c^\top x+ \eta \label{eq:2slp_polyhedral} \\
    \text{s.t.}~& \eta \geq \alpha_j + \inner{\beta_j, x} \ \text{for all} \ j \in \set{J} \notag,
\end{align}
where $\set{J}$ represents a bundle or the indices of pieces in the polyhedral function $\expect{Q(x,\rv)}{\PP}$. Notice that the above form is also an LP problem; hence, we can apply the PLDC policy \eqref{eq:pldc_policy}. However, there are two challenges in applying the process described in \S\ref{sec:pldc_lp}. Firstly, the classical decomposition methods for solving 2-SLPs, such as the L-Shaped \citep{Van.Wets-1969}, proximal bundle \citep{Ruszczynski-1986} and SD \citep{Higle.Sen-1994} find a local approximation of the polyhedral function. In other words, only a subset of $\set{J}$ is discovered. We compute exact pieces of the polyhedral function in the L-Shaped method, and stochastic approximations of the pieces in the SD algorithm.  While the convergence guarantees of the respective algorithms validate such approximations, the policy design procedures associated with a specific solution algorithm must account for the nature of the pieces computed by those algorithms. The second challenge lies in the policy's design steps. Recall that we need the indices of basic variables to construct the cells. Since decomposition algorithms generate local approximations, we can recover different parts of the polyhedral functions at distinct training data points. We have to account for these differences when choosing the basic variable indices that determine how to partition the right-hand sides of the training dataset into distinct cells. To address these challenges, we present tailored approaches for constructing the cells and the policy that account for the specific features of the solution algorithm (L-Shaped or SD) used to obtain the training data. 

\subsection{Description of Policy Design using L-Shaped} \label{sec:lshapedpolicy}
We refer the reader to \cite{Van.Wets-1969,ruszc03bk,shapiro2021lectures} for a detailed description of the L-Shaped method and begin by directly discussing the relevant information collected and used from the L-Shaped runs to design the policy. 

The L-Shaped method uses an outer linearization-based approximation of the first-stage objective function computed at a sequence of candidate solutions. The function approximation is a pointwise maximum of affine pieces. The slope and intercept of an individual piece are $ \alpha = \sum_{\obs \in \rvset} p^\obs \inner{\pi(\obs),h(\obs)}$ and $ \beta  = - \sum_{\obs \in \rvset} p^\obs(\pi(\obs))^\top T(\obs)$, respectively, where $p^\obs$ is the probability of scenario $\obs$ and $\pi(\obs)$ is an optimal dual solution of scenario subproblem computed at a candidate solution. Note that the L-Shaped operates under a complete knowledge of the probability distribution $(p^\obs)_{\obs \in \rvset}$, and therefore, the objective function approximation formed in the algorithm is deterministic. Moreover, using the optimal dual solutions of the subproblem, the affine pieces serve as exact supporting hyperplanes of the expectation-valued objective function along the candidate solution trajectory. If $\tau$ denotes the terminal iteration of an L-Shaped run, then the first-stage problem upon termination takes the form in \eqref{eq:2slp_polyhedral}, with a bundle $\set{J}^\tau \subseteq \set{J}$ of all affine pieces discovered in the $\tau$ iterations of the run. The optimal solution of the resultant problem satisfies:
\begin{align}
& \underbrace{\begin{bmatrix} 
    A_B & \mathbf{0} & \mathbf{0}_{\vert J^\tau \vert} \\ 
    -\bfbeta & \boldsymbol{1}  & -\mathbf{I}_{\vert J^\tau \vert} 
\end{bmatrix}}_{\coloneqq \bfB}
\begin{bmatrix} x^\star_B \\ \eta^\star \\ \zeta^\star_B \end{bmatrix} 
= \begin{bmatrix} b \\ \boldsymbol{\alpha}
\end{bmatrix}, \label{eq:2slpbasis}
\end{align}
where $A_B$ refers to the matrix $A$ restricted to columns corresponding to basic variables in $x^\star$ and $\zeta^\star_B$ are the optimal basic surplus variables corresponding to inequality constraints. We gather the intercept terms into $\bfalpha = (\alpha_1, \hdots, \alpha_{\vert J^\tau \vert})^\top$ and slope vectors into $\bfbeta = (-\bfbeta_1^\top; \hdots; -\bfbeta^\top_{\vert J^\tau \vert})$. Notice that the bundle may include pieces generated in all the iterations or a subset of pieces resulting from a bundle management scheme. For instance, if we employ the regularized variant of the L-Shaped method, the bundle may contain only the pieces active at the terminal solution \citep{Ruszczynski-1986}. Nonetheless, our PLDC policy design procedure applies regardless of the bundle management scheme utilized. 

\subsubsection{Creating Cells}
In deterministic LPs, the basis size remains the same even as we vary the right-hand sides. Since we assume that $A$ is of full rank, the resulting bases span the $\RR^{m_1}$-dimensional space. In contrast, the terminal first-stage problem may have a different number of cuts in the bundle when 2-SLPs are solved with different right-hand sides. Consequently, the resulting bases may span spaces of different dimensions. Even when the bundle sizes are the same, the cuts within each bundle may differ, introducing distinct surplus variables (corresponding to the cuts included as inequality constraints in the terminal first-stage problem) into the optimal basis. Therefore, it is inappropriate to build cells using the basis in \eqref{eq:2slpbasis}.

To address this challenge, we utilize a consolidated master problem. We denote by $\cA_i$ the set of active cuts at the terminal solution of the 2-SLP with right-hand side $b = b^i$. We assemble the unique active cuts generated from solving 2-SLPs with $n$ different right-hand sides into $\cA = \cup_{i=1}^n \cA_i$. Using the set $\cA$, we build the consolidated master problem
\begin{align}  \label{eq:consolmasterproblem}
    \min_{x \in \set{X},\eta}~& c^\top x+ \eta \tag{$CM$}\\
    \text{s.t.}~& \eta \geq \alpha_j + \inner{\beta_j, x} \ \text{for all} \ (\alpha_j,\beta_j) \in \cA. \nonumber
\end{align}
Notice that $f(x) \geq \max_{(\alpha_j,\beta_j) \in \cA}\{ \alpha_j + \inner{\beta_j,x}\}$ for all $x \in \set{X}$. Therefore, solving the consolidated master at $b^i \in \widehat{\set{B}}$ recovers an optimal solution, although it may differ from the one at termination of the L-Shaped method. Alternatively, we can use the terminal solution $(x^\star(b^i),\eta^\star(b^i))$ to recover the basis of the consolidated master \eqref{eq:consolmasterproblem} for all $i \in [n]$. Consequently, we identify the indices of the optimal basic variables and use them to create the collection of cells $\{\cC^\ell\}_{\ell = 1}^L$. 

\subsubsection{PLDC Policy}
Through the master problem \eqref{eq:consolmasterproblem} and associated cells, we modify our PLDC policy \eqref{eq:pldc_policy} for 2-SLP to
\begin{align}
\phi(\rhsobs) & = \max_{\ell \in [L]} u^\ell \left(\begin{pmatrix} \rhsobs \\ \bfalpha \end{pmatrix} - \begin{pmatrix} b^\ell \\ \bfalpha \end{pmatrix} \right) + \begin{pmatrix} x^\star(\rhsobs^\ell) \\ \eta^\star(b^\ell) \end{pmatrix} + z^\ell  \nonumber \\ 
& \qquad - \max_{\ell \in [L]} v^{\ell} \left(\begin{pmatrix} \rhsobs \\ \bfalpha \end{pmatrix} - \begin{pmatrix} b^\ell \\ \bfalpha \end{pmatrix} \right) + z^\ell. \label{eq:pldc_2slp}
\end{align}
The parameters $u^\ell, v^\ell$ and $z^\ell$ of policy \eqref{eq:pldc_2slp} are obtained from solving the training problem 
\begin{subequations} \label{eq:training_problem2slp}
\begin{align}
& \min_{u,v,z}  \sum_{l=1}^{L} \Big(\sum_{k =1}^{d_x+1} \Vert u^\ell_k\Vert_1 + \Vert v^\ell_k \Vert_1\Big) + \sum_{i=1}^n \Vert z^i\Vert_1 \Big)  \\
& \text{s.t.} \ \ u^{\ell(j)}  \left(\begin{pmatrix} \rhsobs^i \\ \bfalpha \end{pmatrix} - \begin{pmatrix} b^j \\ \bfalpha \end{pmatrix} \right) + \begin{pmatrix} x^\star(b^j) \\ \eta^\star(b^j) \end{pmatrix} + z^j \nonumber \\
& \qquad \leq \begin{pmatrix} x^\star(b^i) \\ \eta^\star(b^i) \end{pmatrix} + z^i   \ \text{for all} \ i,j \in [n], \\
& v^{\ell(j)} \left(\begin{pmatrix} \rhsobs^i \\ \bfalpha \end{pmatrix} - \begin{pmatrix} b^j \\ \bfalpha \end{pmatrix} \right) + z^j \leq  z^i  \ \text{for all} \ i,j \in [n],
\end{align}
\end{subequations}
where the inequalities are defined componentwise, $u^\ell, v^\ell \in \RR^{(d_x +1) \times (m_1 + \vert \cA \vert)}$, and $z^\ell \in \RR^{d_x +1}$. The policy $\phi(b)$ returns a first-stage decision $x$, hereafter referred to as a policy-prescribed solution, and an approximated value $\eta$ of the polyhedral function. 

The updated policy in \eqref{eq:pldc_2slp} inherits the properties of \eqref{eq:pldc_policy} identified in the previous section, including the vital property unfolded in Theorem \ref{thm:pldc_properties}, i.e., $\phi(b) = (x^\star(b),\eta^\star(b)) \ \text{for all} \ b \in \text{Conv}(\cC^\ell), \ \ell \ \in \ [L]$. Moreover, for any $b \in \set{B}$, if $\phi(b) = (\hat{x}(b),\hat{\eta}(b))$, the policy-prescribed solution and the approximate objective function value, if the policy recovers an optimal piece over a cell, the epigraph variable $\hat{\eta}(b)$ returned by the policy is the exact first-stage value at the policy-prescribed solution and is optimal over that cell. When the policy-prescribed solution is only feasible, the corresponding value $\hat{\eta}(b)$ may lie above the polyhedral function. In this case, we must estimate the objective function value as $\max_{(\alpha_j,\beta_j) \in \cA} \{\alpha_j + \inner{\beta_j,\hat{x}(b)} \}$, which is a lower bound on the true value $f(\hat{x}(b))$.
The entire policy construction procedure is summarized in Algorithm \ref{alg:pldc_design_2slp}, which we call a static procedure because it uses a fixed number of right-hand sides.
\begin{algorithm}[!t]
\caption{Static Procedure: Decomposition Algorithm-guided Policy for 2-SLP}
\label{alg:pldc_design_2slp}
\begin{algorithmic}[1]
    \State \textbf{Input:} $\widehat{\set{B}} = \{b^i\}_{i=1}^n, \widehat{\set{X}} = \{(x^\star(b^i),\eta^\star(b^i))\}_{i=1}^n$, set of cuts $\cA$
    \State Construct the consolidated master problem \eqref{eq:consolmasterproblem} using $\cA$\label{stepcmp}
    \State For every $b^i$, find  the indices of basic variables, $B^i$, associated with $(x^\star(b^i),\eta^\star(b^i))$. Let the number of such sets be $L (\leq n)$
    \State Gather right-hand sides having same basic variables and place them into the same cell, i.e., $\cC^\ell = \{b^{i_1},b^{i_2},\ldots, b^{i_{m_\ell}} \}$ such that $B^{i_1} = B^{i_2} = \cdots \cdots = B^{i_{m_\ell}}$
    \State Pick one right-hand side $b^\ell$ and corresponding optimal solution $(x^\star(b^\ell),\eta^\star(b^\ell))$ from each cell $\cC^\ell$
    \State Solve the training problem \eqref{eq:training_problem2slp} and get its optimal solution $(u^\ell, v^\ell, z^\ell)$
    \State \textbf{Output:} Policy parameters $\{(u^\ell, v^\ell, z^\ell)\}_{\ell = 1}^L$ and policy input $\{b^\ell, (x^\star(b^\ell),\eta^\star(b^\ell)) \}_{\ell = 1}^L$.
\end{algorithmic}
\end{algorithm}

When the scenario size is large, or the support is a continuous set, we can apply the L-Shaped method to the sample average approximation (SAA) of problem \eqref{eq:2slp}. In this case, our PLDC policy corresponds to the SAA problem. Several studies examine the quality of the SAA solution, and a variety of sequential procedures are used to measure it (for example, \cite{Bayraksan.Morton-2006, Bayraksan.Morton-2011}). These methods can also be used to assess the policy-prescribed SAA solutions. The connection between the PLDC policy of SAA and the true policy is beyond the scope of this paper and warrants further research. Alternatively to the SAA, one can use the SD method as the solution algorithm to train the PLDC policy, which we describe next. 

\subsection{Description of Policy Design using Stochastic Decomposition}
SD \citep{Higle.Sen-1994} is an elegant method for solving stochastic programs, even when the support is continuous. We provide a high-level description of SD sufficient to illustrate the policy construction process and refer the reader to \cite{Sen.Liu-2016} for a detailed exposition on the recent version. 

Like L-shaped, SD is also an outer-linearization approach that maintains a bundle of affine functions to approximate the first-stage expectation-valued objective. SD exhibits three important features that make it suitable for 2-SP problems with either continuous support or a large number of scenarios. First, it incorporates concurrent sampling and optimization steps that operate on an iteratively updated sample-average function, thereby avoiding the computationally intensive approach of solving a sequence of SAA problems. Secondly, SD efficiently re-utilizes previously discovered solution information to compute lower-bounding affine pieces of the sample-average functions. In other words, SD does not solve all past scenario subproblems to optimality; thus, it speeds up each iteration. Moreover, the affine pieces generated by SD are stochastic in nature, and they support the expectation-valued objective function only asymptotically. Finally, SD employs a regularization term in the first-stage approximate problem at every iteration. It allows SD to retain only a finite number of affine pieces in the bundle, further reducing the computational burden. The regularization term uses a sequence of incumbent solutions, and the algorithm is guaranteed to generate a sequence that converges to an optimal solution with probability one. The terminal first-stage approximate problem takes the following form
\begin{align}  \label{eq:sdmasterproblem}
    \min_{x \in \set{X},\eta}~& f(x) = c^\top x+ \eta + \frac{\varsigma}{2}\Vert x - \hat{x}^{\tau} \Vert^2 \tag{$M^\tau$} \\
    \text{s.t.}~& \eta \geq \hat{\alpha}_j + \inner{\hat{\beta}_j, x} \ \text{for all} \ j \in J^\tau, \nonumber
\end{align}
where $\hat{x}^\tau$ is the incumbent solution that meets the statistical optimality conditions. Due to the stochastic nature of SD, termination is based on a notion of statistical optimality that relies on estimates of primal-dual gap values, among other quantities (see \cite{Sen.Liu-2016} for details). Replicating SD runs with different seeds may yield different numbers of observations at termination, even for a 2-SLP with the same right-hand side. Consequently, the coefficients of the affine functions in the bundle $J^\tau$ are stochastic. Therefore, we must handle random optimal basis matrices, which poses a challenge for grouping the right-hand sides according to the random basic variables to construct the necessary cells needed for our policy design process. 

\subsubsection{Creating Cells}
For a right-hand side $b^i$, let $\hat{x}^i$ be the incumbent solution after the termination of SD. The affine functions active at $\hat{x}^i$ are collected into the set $\cA_i$. Since SD is invoked independently for each right-hand side $b \in \widehat{\set{B}}$, the number of second-stage observations used in the sample-average of the objective function can vary in each execution. Consequently, the collection $\cA = \cup_{i=1}^n \cA_i$ does not necessarily provide a valid lower bound on the first-stage objective. Thus, the consolidated master problem \ref{eq:consolmasterproblem} that we used to design the L-Shaped driven policy using $\cA$ is not meaningful with respect to the original problem \eqref{eq:2slp}. We address this issue by constructing out-of-sample affine functions at all the identified incumbent solutions $\{\hat{x}^i\}_{i=1}^n$. The workflow for generating such affine functions is as follows.
\begin{enumerate}
    \item Draw a sample $\{\obs^t\}_{t\in [N]}$ independently from the scenarios observed in $n$ solves of 2-SLPs.
    \item At an incumbent solution $\hat{x}^i$, solve the scenario subproblem for each $\obs \in \{\obs^t\}_{t\in [N]}$. Using the optimal dual solutions $\{\pi^i_t\}_{t\in [N]}$ of the of scenario subproblem, 
    we construct affine functions with coefficients
        \begin{align}
            \alpha_i & = \frac{1}{N}\sum_{t=1}^N (\pi^i_t)^{\top} h(\obs^{t}); \ \  \beta_i = -\frac{1}{N}\sum_{t=1}^N (\pi^i_t)^{\top}T(\obs^{t}).  \label{eq:outofsamplepieces}
        \end{align}
    We refer to them as out-of-sample affine functions. 
    \item Collect the out-of-sample affine functions into a bundle $\widebar{\set{A}}$. 
\end{enumerate}

A few comments about the out-of-sample bundle are in order. Firstly, any piece in the bundle $\widebar{\set{A}}$ lower bounds an out-of-sample function approximation of the expected second-stage value function. That is 
\begin{align}
    \frac{1}{N}\sum_{t=1}^N Q(x,\obs_{t}) & \geq \alpha_i + \inner{\beta_i, x} , \ \text{for all} \ (\alpha_i,\beta_i) \ \in \widebar{\set{A}}. \label{eq:outofsamplelbs}
\end{align}
Secondly, the cardinality of $\widebar{\set{A}}$ is at most $n$ as some of the affine functions might be repeated. Using the central limit theorem, it is evident that the left-hand side in \eqref{eq:outofsamplelbs} converges to the expected recourse value, almost surely, and therefore, the piecewise affine function defined as the pointwise maximum of the affine functions in $\widebar{\set{A}}$ provides a lower bound in expectation. Moreover, as the number of right-hand sides in $\widehat{\set{B}}$ grows to infinity, bundle $\widebar{\set{A}}$ recovers all the affine pieces of polyhedral function $\expect{Q(x,\rv)}{}$ over $\set{B}$. Finally, and most importantly for our purposes, we can formulate the consolidated master problem using affine functions over the bundle $\widebar{\set{A}}$. Thus, the SD-guided policy is obtained by invoking Algorithm \ref{alg:pldc_design_2slp} with $\set{A}$ replaced by $\widebar{\set{A}}$. Since $\hat{x}^i$ is optimal with probability one, there exists $N_i$ such that out-of-sample function value $c^\top \hat{x}^i+\frac{1}{N}\sum_{t=1}^N Q(\hat{x}^i,\obs_{t})$ lies in $\epsilon$-neighborhood of $v^\star(b^i)$ \citep{Higle.Sen-1991,Higle.Sen-1994}. As a result, the basis obtained using the out-of-sample bundle $\widebar{\set{A}}$ with $N\geq \max_i\{N_i\}$, and hence the policy provides solutions that are optimal with probability one.

Irrespective of the solution method employed, Algorithm \ref{alg:pldc_design_2slp} uses a fixed training sample size $n$ to recover seen pieces of the solution trajectory. In practice, however, new right-hand sides may be observed sequentially, raising the question of how to feed the information learned from new solves into the policy and update it. In the next section, we develop a sequential procedure to address this question.

\subsection{Sequential Procedure for Policy Design} \label{sec:seqproc}
In this section, we present a sequential procedure to train the PLDC policy. Suppose that a 2-SLP problem has been solved for a certain number of right-hand sides using a decomposition algorithm. These solves provide optimal bases and the high-quality estimates of the first-stage objective function at their respective terminal solutions. The sequential procedure is designed to assess the quality of the policy trained using the information collected from these solves, terminate if the policy-prescribed solutions are of sufficiently high quality, or continue the training process otherwise.

Our procedure performs multiple rounds sequentially, with each round comprising two steps. First, the current policy is evaluated on a new batch of right-hand sides. To perform this evaluation, we implement a decomposition algorithm to compare the first-stage objective value at the policy-prescribed solution with the optimal value. The second step appends new solve information, namely, the right-hand vectors, active cuts, corresponding optimal solutions, and estimated objective function values, to the training data. We incorporate only the information from solves whose solutions are rendered infeasible or suboptimal by the current policy. This choice allows us to maintain a training problem of manageable size by retaining only that information guaranteed to improve the quality of the policy. With the updated training data, the procedure solves the policy-training problem and produces an updated policy. The procedure proceeds to the next round and repeats the two steps above. We present a detailed description of the sequential policy update procedure below.

We start with an initial training data, $\widehat{\set{B}}_0 = \{b^i\}_{i=1}^{n_0}, \widehat{\set{X}}_0 = \{x^\star(b^i)\}_{i=1}^{n_0}$ and $\widehat{\set{V}}_0 = \{v^\star(b^i)\}_{i=1}^{n_0}$. The set of corresponding active cuts is denoted by $\widehat{\set{A}}_0$. If $\widehat{\set{X}}_0$ and $\widehat{\set{V}}_0$ are obtained using the L-Shaped method, $\widehat{\set{A}}_0 = \set{A}_0$; if they are obtained using SD, $\widehat{\set{A}}_0 = \widebar{\set{A}}_0$. The static procedure in Algorithm \ref{alg:pldc_design_2slp} is invoked on this data to arrive at an initial policy $\phi^1$. At round $t \geq 1$, we generate a batch of right-hand side observations $\widehat{\set{B}}_t = \{b^1,b^2,\ldots,b^{n_t}\}$ independently from those observed in previous rounds. The policy is applied to each of these sampled right-hand sides, and the solution $\phi^t(b^i) = (\hat{x}^t(b^i),\hat{\eta}^t(b^i))$ is recorded. We execute the stage decomposition method (L-Shaped or SD) on each of the chosen right-hand sides and collect the solutions information $\widehat{\set{X}} = \{x^\star(b^i)\}_{i=1}^{n_t}$ and $\widehat{\set{V}} = \{v^\star(b^i)\}_{i=1}^{n_t}$, and the bundle of active cuts $\widehat{\set{A}}_t$. Next, we compute the fraction of policy-prescribed solutions that are feasible to the corresponding 2-SLP instance.  For this, we check the feasibility of the policy prescribed solution $\hat{x}^t(b^i)$ up to a tolerance level $\epsilon_\set{X}$ (as is common practice in solvers). To assess this, we determine the fraction of infeasible instances as
\begin{align} \label{eq:infeasinst}
R_{\set{X}}(\phi^t) = 1 - \expect{\mathbbm{1}(A\hat{x}^t(\rhsrv) =  \rhsrv + \epsilon_{\set{X}}, \hat{x}^t(\rhsrv) \geq 0)}{\rhsrv}.
\end{align}
We estimate the quantity in \eqref{eq:infeasinst} and its $(1-\nu)$-confidence interval using the batch $\widehat{\set{B}}_t$ as: 
\begin{subequations}    
\begin{align}
    R_{\set{X}}^{n_t}(\phi^t) =~& 1-  \frac{1}{n_t} \sum_{i=1}^{n_t} \mathbbm{1}(A\hat{x}(b^i) = b^i + \epsilon_{\set{X}}, \hat{x}^t(b) \geq 0 ), \\
    R_{\set{X}}(\phi^t) \in~& \Bigg(0,  R^{n_t}_{\set{X}} + z_{\nu/2} \frac{1}{\sqrt{4n_t}} \Bigg). \label{eq:infeas_confInt}
\end{align}
\end{subequations}
Whenever the interval length in \eqref{eq:infeas_confInt} is sufficiently small, we can conclude that $\phi^t$ generates feasible solutions to 2-SLP instances for almost every right-hand side with high confidence. Hence, at a round $t$, if $n_t$ is such that the interval length is below a threshold $\rho$, then we proceed to assessing the quality of the objective function value at the policy-prescribed solution.

We carry out the objective function value assessment by revisiting the nonconvex problem \eqref{eq:risk_neutral}. We permit the user-defined errors in the policy-prescribed solution and rewrite problem \eqref{eq:risk_neutral} into the following form
\begin{align} \label{eq:optimalpol_ind}
\min_{\phi \in \hat{\Phi}_{\text{PL}}}~& \bigg \{p(\phi) \coloneqq \expect{\mathbbm{1}\{f(\hat{x}(\rhsrv)) - v^\star(\rhsrv) > \epsilon_f \}}{\rhsrv} \bigg \}, 
\end{align} 
where $\hat{\Phi}_{\text{PL}} = \{ \phi \in \Phi_{\text{PL}} \ \vert \ \phi(b) = (\hat{x}(b),\hat{\eta}(b)), \  A\hat{x}(b) = b + \epsilon_\set{X}, \ \hat{x}(\rhsobs) \geq 0 \ \text{for almost every} \ \rhsobs \in \cB  \}$.
The value $p(\phi)$ represents the fraction of instances that are $\epsilon_f$-away from their optimal value at the policy-prescribed solution. Since $\phi^\star$ is an element of $\hat{\Phi}_{\text{PL}}$, in which case the expectation in \eqref{eq:optimalpol_ind} evaluates to zero, we have $p(\phi) \geq 0$ for all $\phi \in \hat{\Phi}_{\text{PL}}$. Therefore, at round $t$, if $p(\phi^t) \leq \epsilon$ given that $\phi^t(b)$ is feasible for almost every $b \in \cB$, we say that $\phi^t$ is a good quality policy. However, the value of $p(\phi^t)$ is unknown and we have to resort to sampling-based estimation. When the number of scenarios is manageable, we can compute $f(\hat{x}(b))$ and use the value $v^\star(b)$ returned by the L-Shaped method to obtain an estimate of $p(\phi^t)$ over the batch $\widehat{\set{B}}_t$. For the large scenario size, the function value $f(\hat{x}(b))$ in \eqref{eq:optimalpol_ind} needs to be estimated. Moreover, in this case SD is our recommended choice of the algorithm and the value $v^\star(b)$ reported by SD is a random quantity. We therefore present the policy evaluation procedure separately for the L-Shaped and the SD.

\subsubsection{Evaluating L-Shaped-Guided Policy} \label{sec:lshapedpolicyeval} 
At round $t$, we begin by computing an estimate of \eqref{eq:optimalpol_ind} using the batch $\widehat{\set{B}}_t$ as
\begin{align}
    p_{m_t} & = \frac{1}{m_t} \sum_{j \in \cI_{\set{X}}} \mathbbm{1}\{ f(\hat{x}^t(b^j))  - v^\star(b^j) > \epsilon_f \}, \label{eq:lshapednonoptimal}
\end{align}
where $m_t = \vert \cI_\set{X} \vert$ and $\cI_\set{X} \subseteq [n_t]$ indexes instances where the policy-prescribed solution is feasible. Noting that $\expect{p_{m_t}}{} = p(\phi^t)$ and $\text{Var}(p_{m_t}) = \frac{p(\phi^t)(1-p(\phi^t))}{m_t} \leq \frac{1}{4m_t}$, we obtain the one-sided confidence interval on $p(\phi^t)$ at level $(1-\mu)$ as
\begin{align}
  & p(\phi^t) \in \Bigg(  0, p_{m_t} + z_{\mu/2}\frac{1}{\sqrt{4m_t}} \Bigg), \label{eq:nonopt_confInt}
\end{align}
We conclude that when the feasible instances at the current policy $\phi^t(\cdot)$ meet the condition $p_{m_t} + z_{\mu/2}\frac{1}{\sqrt{4m_t}} \leq \epsilon$, we terminate the sequential procedure. 

\subsubsection{Evaluating SD-Guided Policy} When we are dealing with 2-SLPs with significantly large or continuous support, computing the gap $f(\hat{x}^t(b)) - v^\star(b)$ exactly, even for a given $b$, is not possible. For such problems, our proposed solution approach, SD, is terminated based on statistical optimality conditions. Therefore, we only have an estimate of $v^\star(b)$ computed using a sample encountered during the execution of SD. The function value $f(\hat{x}^t(b))$ can only be estimated using sampling. 

To examine the similarities between $f(x^\star(b))$ and $f(\hat{x}^t(b))$ for a given $b \in \widehat{B}_t$, we test the null hypothesis $\text{H}_0(b): f(x^\star(b)) = f(\hat{x}^t(b))$ against $\text{H}_1(b): f(x^\star(b)) \neq f(\hat{x}^t(b))$. We estimate $f(x^\star(b))$ and $f(\hat{x}^t(b))$ using the same sample of $M$ observations, $\obs^1,\ldots,\obs^M$, generated independently from the observations used in the SD execution. Employing the common random numbers to estimate $f(x^\star(b))$ and $f(\hat{x}^t(b))$ ensures a low-variance estimation of the 
gap $f(x^\star(b)) - f(\hat{x}^t(b))$. Notice that the sample used can be the same as that used to generate the out-of-sample affine function in \eqref{eq:outofsamplepieces}, in which case we only have to estimate $f(\hat{x}^t(b))$. If the scenario-specific objective function values are equal for all observations $\obs^1,\ldots,\obs^M$, then we trivially accept the null hypothesis. Otherwise, we compute the sample mean estimate of the gap $f(x^\star(b)) - f(\hat{x}^t(b))$, denoted by $\widebar{G}$, and its sample standard deviation $S_g$. We reject the null hypothesis if the t-statistic, $\sqrt{M}\widebar{G}/S_g$, is greater than $t_{M-1, \nu/2}$, where $t_{M-1, \nu/2}$ is the critical value at significance level $\nu$ and $M-1$ degrees of freedom. Otherwise, we accept the null hypothesis, which implies that the objective values at the policy-prescribed and SD solutions are not significantly different. Specifically, the difference $f(x^\star(b)) - f(\hat{x}^t(b))$ lies in the interval $\bigg(\widebar{G} - t_{M-1, \nu/2} \frac{S_g}{\sqrt{M}},\widebar{G} + t_{M-1, \nu/2}\frac{S_g}{\sqrt{M}}\bigg)$ with $100(1-\nu)$ confidence. Thus, we conclude that the policy-prescribed solution is of good quality. 

Analogous to $p_{m_t}$ defined in section \ref{sec:lshapedpolicyeval}, we denote by $\hat{r} = \expect{\mathbbm{1}(\text{H}_0(\rhsrv) \ \text{is rejected})}{\rhsrv}$, the fraction of instances for which the null hypothesis is rejected (the policy-prescribed solution is not acceptable). Its sample mean estimator is computed as
\begin{align}
r_{m_t} & = \frac{1}{m_t} \sum_{j \in \cI_{\set{X}}} \mathbbm{1}( \text{H}_0(\rhsobs^j) \ \text{is rejected} ), \label{eq:sdnonoptimal}
\end{align}
where $\cI_{\set{X}}$ is the set of instances feasible at the policy-prescribed solution. Then the one sided $100(1-\mu)$ confidence interval is given by
\begin{align}
\Bigg(  0, r_{m_t} + z_{\mu/2}\frac{1}{\sqrt{4m_t}} \Bigg), \label{eq:sdnonopt_confInt}
\end{align}
With this, we conclude the evaluation of the policy-prescribed solution.

Using the interval in \eqref{eq:infeas_confInt} to determine the feasibility of the policy-prescribed solutions to the 2-SLP problem, and \eqref{eq:nonopt_confInt} (for L-Shaped), \eqref{eq:sdnonopt_confInt} (for SD) to determine the quality of the objective function values, we decide whether to terminate the sequential procedure or continue refining the policy. Specifically, if the lengths of these intervals fall within user-defined thresholds ($\rho \in (0,1)$ and $\epsilon > 0$), the procedure is terminated and outputs the current policy $\phi^t$; otherwise, we append the suboptimal and infeasible instances into the previous training data, $\widehat{\set{B}}_{t-1}, \widehat{\set{X}}_{t-1}$ and add the set of cuts $\widehat{\set{A}}_t$ into the bundle $\widehat{\set{A}}_{t-1}$. The algorithm \ref{alg:pldc_design_2slp} is implemented on this updated data to get the policy $\phi^{t+1}$ and completes round $t$. We move to the next round $t+1$ with a new batch of samples of size $n_{t+1} = \varrho n_{t} $, where $ \varrho > 1$, and repeat the above steps. The pseudocode of our sequential procedure explained above is presented in Algorithm \ref{alg:seqproc_lshaped}. 
\begin{algorithm}[!ht]
\caption{Sequential Procedure: Decomposition Algorithm-guided Policy for 2-SLPs}
\label{alg:seqproc_lshaped}
\begin{algorithmic}[1]
\State \textbf{Input:} $n_0 \in \mathbb{Z}_+$ , $\widehat{\set{B}}_0 = \{b^i\}_{i=1}^{n_0}, \widehat{\set{X}}_0 = \{(x^\star(b^i),\eta^\star(b^i))\}_{i=1}^{n_0}$, $\widehat{\set{V}}_0 = \{v^\star(b^i)\}_{i=1}^{n_0}$, $\widehat{\set{A}}_0$, $\epsilon > 0, \epsilon_\set{X} > 0, \epsilon_f > 0$, $t = 1$, $\rho \in (0,1), \varrho > 1, n_1 = \lceil \varrho n_0 \rceil$, $z_{\mu/2}$-score of $100(1-\mu)$ confidence interval
\State \textbf{Initial policy:} Construct policy $\phi^1$ using algorithm \ref{alg:pldc_design_2slp} with inputs $\widehat{\set{B}}_0, \widehat{\set{X}}_0$ and $\widehat{\set{A}}_0$
\State \label{drawrhs_ls} Draw a sample $\widehat{\set{B}}_t = \{ b^1,b^2,\ldots,b^{n_t} \}$ independently and identically distributed
\State Compute solution $\phi^t(b^i) = (\hat{x}^t(b^i),\hat{\eta}^t(b^i))$ for each $b^i \in \widehat{\set{B}}_t$
\State Calculate fraction of infeasible instances:
\begin{align} 
R_{\set{X}}^{n_t}(\phi^t) = 1-\frac{1}{n_t} \sum_{i=1}^{n_t}  \mathbbm{1}(A\hat{x}^t(b^i) = b^i + \epsilon_{\set{X}}, \hat{x}^t(b^i) \geq 0 )
\end{align}
\If{$R^{n_t}_{\set{X}} + \frac{z_{\nu/2}}{2\sqrt{n_t}} \leq \rho$}
\State Collect indices of feasible $\hat{x}^t(b^i)$ into set $\cI_\set{X} = \{i_1,i_2,\ldots,i_{m_t}\}$
\State Run L-Shaped (or SD) and obtain $\widehat{\set{X}}_t = \{(x^\star(b^i),\eta^\star(b^i))\}_{i=1}^{n_t}$ and $\widehat{\set{V}}_t = \{v^\star(b^i)\}_{i=1}^{n_t}$ along with bundle of cuts $\widehat{\set{A}}_t$
\State Compute fraction of suboptimal instances: $p_{m_t}$ using \eqref{eq:lshapednonoptimal} for L-Shaped (or $r_{m_t}$ using \eqref{eq:sdnonoptimal} for SD)
\If{$p_{m_t} + \frac{z_{\mu/2}}{2\sqrt{m_t}} \leq \epsilon$ $\big(\text{or} \ r_{m_t} + \frac{z_{\mu/2}}{2\sqrt{m_t}} \leq \epsilon \big)$}
\State output $\phi^t$ and terminate the procedure
\Else
\State go to step \ref{appenddata_ls}
\EndIf
\Else
\State go to step \ref{appenddata_ls}
\EndIf
\State \label{appenddata_ls} Collect indices of infeasible and suboptimal instances into $\cI^c_\set{X} \subseteq [n_t]$ and $\cI_{f} \subseteq \cI_\set{X}$ respectively. Reset $\widehat{\set{B}}_t \leftarrow \widehat{\set{B}}_{t-1} \cup \{b^i\}_{i \in \cI^c_\set{X} \cup \cI_f}, \widehat{\set{X}}_t \leftarrow \widehat{\set{X}}_{t-1} \cup \{(x^\star(b^i), \eta^\star(b^i)\}_{i \in \cI^c_\set{X} \cup \cI_f}, \widehat{\set{V}}_{t} \leftarrow \widehat{\set{V}}_{t-1} \cup \{v^\star(b^i)) \}_{i \in \cI^c_\set{X} \cup \cI_f}$ and $\widehat{\set{A}}_t \leftarrow \widehat{\set{A}}_{t-1} \cup \widehat{\set{A}}_t$
\State \label{pldc_param_ls} Invoke algorithm \ref{alg:pldc_design_2slp} with inputs $\widehat{\set{B}}_t,  \widehat{\set{X}}_t$,  $\widehat{\set{A}}_t$
\State Compute policy $\phi^{t+1}$ using the output, $\{(u^\ell, v^\ell, z^\ell)\}_{\ell=1}^{L_t}$ and $\{b^\ell, (x^\star(b^\ell),\eta^\star(b^\ell)) \}_{\ell = 1}^{L_t}$, of algorithm \ref{alg:pldc_design_2slp}
\State Set $n_{t+1} = \varrho n_t $, $t \leftarrow t+1$ and go to step \ref{drawrhs_ls}
\end{algorithmic}
\end{algorithm}

The inclusion of the optimal instances shows that the policy embeds their optimal bases. Adding suboptimal and infeasible instances to the training data passes their information to the policy construction, and we include those bases in the policy's computation. This way, we sequentially add the new pieces to the policy, which significantly helps control the size of the training problem \eqref{eq:training_problem2slp} and the number of pieces in the policy. It is reasonable to expect that our procedure in Algorithm \ref{alg:seqproc_lshaped} terminates after a finite number of rounds because the number of pieces in the optimal policy $\phi^\star$ is finite and each piece can be recovered using a finite number of right-hand sides. To formalize this claim, we define the stopping times at which the procedure attains the desired levels of feasibility and optimality as follows. 
\begin{subequations}
\begin{align}
    T_{\set{X}}(\rho) & = \inf_{t \geq 1} \{ t \ \vert \ R_{\set{X}}^{n_t}(\phi^t) + \frac{z_{\mu/2}}{\sqrt{4n_t}} \leq \rho \},\\
    T_{f}(\epsilon) & = \inf_{t \geq 1} \{ t \ \vert \ p_{m_t}(\phi^t) + \frac{z_{\mu/2}}{\sqrt{4m_t}} \leq \epsilon \}.
\end{align}
\end{subequations}
The stopping time $T_{\set{X}}(\rho)$ is the round index at which the average number of infeasible instances with marginal error is below a threshold $\rho$, for the first time. Analogously, the stopping time $T_{f}(\epsilon)$ represents the first time where the average number of suboptimal instances with marginal error lies below a threshold $\epsilon$. When Algorithm \ref{alg:seqproc_lshaped} employs SD, $T_f(\epsilon)$ can be modified by replacing $p_{m_t}(\phi^t)$ by $r_{m_t}(\phi^t)$. With these stopping times, we establish the finite round convergence of Algorithm \ref{alg:seqproc_lshaped} in Theorem \ref{thm:finiteroundconv} and defer its proof to Appendix \ref{app:finiteroundproof}.

\begin{theorem}{(Finite Round Convergence)} \label{thm:finiteroundconv} Let $(\cB,\Sigma_\rhsobs,\mathbb{Q})$ be a probability space, where $\mathbb{Q}(b) > 0$, almost surely. Let the support set $\cB$ be decomposed into $\bar{L}$ cells, each corresponding to an optimal basis. In every cell $\ell$, suppose there exists a set $\{b^{\ell,i_0},b^{\ell,i_1},\ldots,b^{\ell,i_{m_1}}\}$ of $m_1 +1$ vectors such that $b^{\ell,i_1}-b^{\ell,i_0},\ldots,b^{\ell,i_{m_1}}-b^{\ell,i_0}$ are linearly independent. Let $\phi^t$ be a policy at round $t$ of Algorithm \ref{alg:seqproc_lshaped}. Then, the following results hold:
\begin{enumerate}
    \item If Algorithm \ref{alg:seqproc_lshaped} invokes L-Shaped at each round, then $\mathbb{P}(T_\set{X}(\rho) < \infty) = 1$ and $\mathbb{P}(T_f(\epsilon) < \infty) = 1$.
    \item Suppose SD is deployed at each round of Algorithm \ref{alg:seqproc_lshaped}. If the number of samples $N$ used to construct out-of-sample function approximation (see \eqref{eq:outofsamplelbs}) is sufficiently large, then $\mathbb{P}(T_\set{X}(\rho) < \infty) = 1$ and $\mathbb{P}(T_f(\epsilon) < \infty) = 1$.
\end{enumerate}
\end{theorem}

Theorem \ref{thm:finiteroundconv} states that the sequential procedure terminates in a finite number of rounds for L-Shaped and SD algorithms. This convergence is guaranteed under the assumption that every cell contains $m_1+1$ vectors such that the difference vectors are linearly independent. Such an assumption is crucial for recovering pieces of the optimal policy. The existence of this follows from the Basis Decomposition Theorem of \cite{walkup1969lifting}. The following remarks clarify this point.
\begin{remark}
    From the Basis Decomposition Theorem, the edges of a closed cell $\cC$ are the $m_1$ linearly independent columns of the constraint matrix $A$. These columns form an optimal basis $\bfB(\cC)$ of $\cC$. Clearly, $\cC$ contains these $m_1$ linearly independent columns of $A$. Let these columns be denoted by $a^{i_{1}},\ldots,a^{i_{m_1}}$. Pick any $b^{i_0} \in \cC$ such that $b^{i_0} \neq a^{i_{j}}$ for all $j \in [m_1]$. Then, compute $b^{i_0} + a^{i_{j}} =: b^{i_j}$. It is evident that $b^{i_1} - b^{i_0},\ldots,b^{i_{m_1}} - b^{i_0}$ are linearly independent. Because $b^{i_0}$ and $a^{i_j}$ lie in a closed convex polyhedral cone $\cC$ formed by the nonnegative linear combinations of the columns $a^{i_1},\ldots,a^{i_{m_1}}$ of A, the sum vector $b^{i_j}$ also belongs to $\cC$. In essence, $\cC$ contains $b^{i_0},b^{i_0} + a^{i_1},\ldots,b^{i_0}+a^{i_{m_1}}$ which in turn certifies the existence assumption stated in Theorem \ref{thm:finiteroundconv}.
\end{remark}
\begin{remark} 
    The sampling procedure in Algorithm \ref{alg:seqproc_lshaped} can be adapted so that cell $\cC$ contains $b^{i_0},b^{i_0} + a^{i_1},\ldots,b^{i_0}+a^{i_{m_1}}$. For instance, suppose the procedure discovers a new basis $\bfB^\ell$ for some $ b'$. We can utilize the corresponding columns of the matrix $A$ to construct the set satisfying the required condition stated in Theorem \ref{alg:seqproc_lshaped}. Note that the number of cells and the corresponding optimal basic variables are not known a priori over the set $\set{B}$. As a result, we cannot construct cells with elements $b^{i_0},b^{i_0} + a^{i_1},\ldots,b^{i_0}+a^{i_{m_1}}$ beforehand. The right-hand sides are observed sequentially, and whenever we observe a new basis, we can use its column vectors to incorporate the required set into the cell.
\end{remark}

\section{Computational Study} \label{sec:experiments}
In this section, we present the results from our numerical experiments where we apply the PLDC policy obtained from the static procedure in Algorithm \ref{alg:pldc_design_2slp} and the sequential procedure in Algorithm \ref{alg:seqproc_lshaped}. We conduct these experiments on five well-known 2-SLP problems whose description is provided in Appendix \ref{sec:testinstances}. All experiments were implemented in Julia. Based on the size of the training problem, we conducted experiments on two machines: (a) 64-bit Intel Core i7-1270P x 16 with 32 GB memory, and (b) AMD EPYC 7763 with 128 CPU cores and approximately 500GB of RAM. We begin by specifying our procedures, including the generation of right-hand sides, the number of perturbed constraints, and the feasibility and optimality tolerances. We then present the detailed results for the test problems. 

\subsection{Training-data generation and validation}
A right-hand-side vector and associated optimal solution from each cell are input to our PLDC policy \eqref{eq:pldc_2slp}. For instances with only a few first-stage constraints (PGP2, CEP, and 20TERM), we generate input data by perturbing all first-stage right-hand sides. On the other hand, for the large-scale problems 4NODE and STORM, we select three and five constraints, respectively, and perturb their right-hand sides. Time series models are often used to estimate parameters that follow a particular trend, accounting for the effects of explanatory variables. In connection to this, we generate components of perturbed right-hand vectors using univariate linear time series model of order one, i.e., the $i$-th right-hand side of $k$-th constraint is computed as $b^i_k = a_{0,k} i + a_{1,k} + e^i_k$, where $a_{0,k}$ and $a_{1,k}$ are the time series parameters and $e^i_k \sim \mathcal{N}(0, \sigma^2_i)$. The parameters $a_{0,k}$ and $a_{1,k}$ are chosen in accordance with the right-hand sides of the perturbed constraints of respective instances. We simulate the time series for $T$ time instants and utilize it in our static and sequential procedures. The Latin Hypercube Sampling (LHS) \citep{mckay-1992} is the second sampling approach employed in our experiments to generate right-hand sides. In this approach, we divide the range of each component of the right-hand-side vector into $T$ non-overlapping intervals. For each component, one value is randomly chosen from each interval. These $T$ values for each component are then randomly paired to create $T$ vectors of dimension $m_1$. LHS uniformly covers the entire range of each component and ensures that values are sampled from each interval within that range. This sampling technique achieves good coverage of the sampling space while reducing the variance of the generated right-hand vectors. 

For the generated right-hand sides, we solve the 2-SLP problem instances using the L-Shaped (or SD) method. The cuts active at the terminal solution of the L-Shaped method are stored for all right-hand sides. If 2-SLPs are solved with SD, the out-of-sample cuts are calculated and stored. With either collection of cuts, we construct the consolidated master problem \eqref{eq:consolmasterproblem} and solve it using Gurobi at each right-hand side. We use the optimal basic feasible solution returned by the solver to construct the cells. The resulting training problem \eqref{eq:training_problem2slp} is modeled in JuMP, a Julia modeling package, and solved with the Gurobi solver. A solution to the training problem provides the policy parameters. If a cell has more than one right-hand side, we pick one of the right-hand sides and the corresponding optimal solution to construct the PLDC policy \eqref{eq:pldc_2slp}.

The policy-prescribed solution is acceptable when it is feasible and optimal within the user-defined tolerance. In light of this, we calculate the number of 2-SLP instances feasible under the policy-prescribed solutions and identify how many of these feasible instances are optimal. In our computational experiments, we solve a total of $T$ 2-SLP instances and use $n$ of them for training. The policy is deployed on the remaining $T-n$ validation instances to measure its performance. The gray rows in all tables represent the performance summary on validation instances, whereas the non-highlighted rows indicate the performance on training instances. While evaluating feasibility, we set the feasibility tolerance to $10^{-6}$ on the absolute feasibility gap, the standard criterion for most linear optimization solvers. The optimality of the feasible solution is assessed using the relative optimality gap, with tolerances of $0.0005$ and $0.01$ for small-scale and large-scale instances, respectively. It should be noted that the policy produces a pair consisting of a first-stage decision $\hat{x}$ and the value of the epigraph variable $\hat{\eta}$ (see the consolidated master problem \eqref{eq:consolmasterproblem}). We verify the feasibility of the first-stage constraints at the solution $\hat{x}$. The decision $\hat{\eta}$ need not be a lower bound on the expected recourse function; therefore, we use the value of the polyhedral function at $\hat{x}$ derived from the cuts present in the consolidated master problem.

\subsection{Numerical Results}
The PLDC policy \eqref{eq:pldc_2slp} prescribes an alternative way to obtain the solutions of a 2-SLP problem whenever a change is observed in the first-stage right-hand sides. Observe that policy construction is a one-time computational effort, whereas obtaining solutions whenever a parameter value of a complex 2-SLP changes requires intensive computations. Hence, using the PLDC policy, a functional characterized by sparse parameters is computationally cheaper. With this note, we present our findings from our numerical experiments conducted on 2-SLP problems described in Table \ref{table:instancesSummary}. We start with a discussion on the policy obtained from the static procedure in Algorithm \ref{alg:pldc_design_2slp} with the L-Shaped method.

\subsubsection{Static Procedure with L-Shaped}
For PGP2 and CEP, we solve $800$ instances using right-hand sides generated by a linear time-series model and the LHS. We use $80\%$ of them for training and use the remaining $20\%$ for validation. Table \ref{tab:ltm_pgp2cep} illustrates the performance of the policy when right-hand sides are generated using a linear time series model. The non-highlighted rows reveal that the policy retains the optimal solutions of training instances that align with the policy's properties stated in Theorem \ref{thm:pldc_properties}. The policy performs well on validation instances, yielding feasible solutions for more than $99\%$ of the problems, all of which are optimal. The number of cells reflects the distinct bases identified by the L-Shaped execution across the $640$ training instances. The results show that the cell count is only a small fraction of the total number of instances, indicating that most right-hand sides have common optimal bases. Consequently, this demonstrates that one need not solve a stochastic program from scratch; instead, one can use the designed policy to obtain a high-quality solution for a new instance. The number of active cuts in Table \ref{tab:ltm_pgp2cep} represents the size of the consolidated set $\mathcal{A}$ and highlights that many instances can share the same optimal pieces.

We find a similar performance for right-hand sides generated via LHS, as evident in Table \ref{tab:lhc_pgp2cep}. However, the percentage of feasible and optimal instances decreases due to the high variability in the right-hand side values. This fact is corroborated by the larger number of cells and active cuts when right-hand sides are generated using LHS, compared to when they are generated using the time-series model. While we should expect $\phi(b^i) = (x^\star(b^i),\eta^\star(b^i))$ for $b^i \in \widehat{\set{B}}$, the training dataset, we observe a slight discrepancy in our results, which can be attributed to numerical approximation errors in solving the training problem. These errors are upto an order $\mathcal{O}(\text{tolerance level})$ in some of the results, e.g., the maximum feasibility gap of a training instance of PGP2 in Table \ref{tab:lhc_pgp2cep} is $\mathcal{O}(10^{-6})$.

\begin{table}[!tb]
\scriptsize
\centering
\begin{subtable}[!t]{0.9\textwidth}
    \resizebox{\textwidth}{!}{%
    \centering
        \begin{tabular}{c!{\vrule width \arrayrulewidth}c c c c c c c}
            \hline
            Instance & Size & \# of cells & \# of active & Feasible & Optimal & Feasibility gap & Optimality gap \\
             Name &  & & cuts & instances & instances &  (Maximum) & (Maximum) \\
            \hline
            \multirow{-1}{*}{PGP2}
             & $640$ & $12$ & $187$ & $100\%$ & $100\%$ & $10^{-6}$ & $0.0003$ \\
            \cellcolor{gray!20} & \cellcolor{gray!20}$160$ & \cellcolor{gray!20}\textendash & \cellcolor{gray!20} \textendash & \cellcolor{gray!20}$99.3\%$ & \cellcolor{gray!20}$100\%$ & \cellcolor{gray!20}$9.5$ & \cellcolor{gray!20}$0.0003$ \\
            \hline
            \multirow{-1}{*}{CEP}
            & $640$ & $3$ & $21$ & $100\%$ & $100\%$ & $10^{-6}$ & $0.0003$ \\
            \cellcolor{gray!20} & \cellcolor{gray!20}$160$ & \cellcolor{gray!20}\textendash & \cellcolor{gray!20}\textendash & \cellcolor{gray!20}$100\%$ & \cellcolor{gray!20}$100\%$ & \cellcolor{gray!20}$10^{-6}$ & \cellcolor{gray!20}$0.0003$ \\
            \hline
        \end{tabular}%
    }
    \caption{right-hand sides generated with linear time series model.} \label{tab:ltm_pgp2cep}
    \end{subtable}
    \hfill
    
    \vspace{0.3cm}
    
\begin{subtable}{0.9\textwidth}
\resizebox{\textwidth}{!}{%
\centering
\begin{tabular}{c!{\vrule width \arrayrulewidth}c c c c c c c}
\hline
Instance & Size & \# of cells & \# of active & Feasible & Optimal & Feasibility gap & Optimality gap \\
 Name &  & & cuts & instances & instances &  (Maximum) & (Maximum) \\
\hline
\multirow{-1}{*}{PGP2}
 & $640$ & $72$ & $416$ & $99.5\%$ & $100\%$ & $3.6 \times 10^{-6}$ & $0.0005$ \\
\cellcolor{gray!20} & \cellcolor{gray!20}$160$ & \cellcolor{gray!20}\textendash & \cellcolor{gray!20} \textendash & \cellcolor{gray!20}$93.1\%$ & \cellcolor{gray!20}$98.6\%$ & \cellcolor{gray!20}$3.98$ & \cellcolor{gray!20}$0.006$ \\
\hline
\multirow{-1}{*}{CEP}
& $640$ & $39$ & $92$ & $100\%$ & $100\%$ & $10^{-6}$ & $0.0003$ \\
\cellcolor{gray!20} & \cellcolor{gray!20}$160$ & \cellcolor{gray!20}\textendash & \cellcolor{gray!20}\textendash & \cellcolor{gray!20}$95\%$ & \cellcolor{gray!20}$98.6\%$ & \cellcolor{gray!20}$231.1$ & \cellcolor{gray!20}$0.0058$ \\
\hline
\end{tabular}%
}
\caption{right-hand sides generated with LHS.}
\label{tab:lhc_pgp2cep}
\end{subtable}
\caption{Feasibility tolerance: $10^{-6}$, Optimality tolerance: $0.0005$.}
\vspace*{-8pt}
\end{table}

The runtime of L-Shaped applied to SAA for 4NODE is high because its second-stage subproblems are time-consuming to solve. So, we solve its SAA for $400$ different right-hand sides generated from a linear time-series model, and we split the instances into training and validation sets with a $50\%-50\%$ split. This process is repeated for four SAA replications. Table \ref{tab:ltm_4node} shows that the policy performs well on $200$ validation instances, with over $90\%$ of these being feasible, and all are indeed optimal. 

Our computational experience with the 20TERM and STORM instances has been noteworthy, yielding compelling observations. The objective of 20TERM exhibits high variability \citep{Sen.Liu-2016} and we observe a similar behavior in our numerical results. The active cuts differ significantly across training instances and lead to a considerable gap in the right-hand-side values of the training problem's constraints, as detailed in the Gurobi coefficient statistics report. This variability leads to an ill-conditioned problem, causing the solver to encounter numerical issues and terminate with the model status set to infeasible. However, the solver logs confirm that the training problem has at least one feasible solution. We address this numerical issue by relaxing the constraints of the training problem. We add a nonnegative auxiliary variable to the left-hand side of each constraint and minimize its magnitude in the objective function. For large-scale and unstable instances such as 20TERM and STORM, we solve the relaxed version of the training problem. These observations reveal another interesting phenomenon that encourages the use of the relaxed training problem. Solutions to SP problems with prohibitively large scenario sizes (or continuous support) are approximate, based on a chosen tolerance, and obtained using a potentially suboptimal basis matrix. Even for right-hand sides in the same cell, different bases may be identified, rendering the training problem infeasible. In these cases, the relaxed form of the training problem can be used to compute the policy efficiently.

We solve $300$ instances of 20TERM and STORM and use $80\%$ of them in the policy construction procedure. The size of the training problem for 20TERM and STORM grows rapidly with the number of data points. To demonstrate the policy's effectiveness, we begin with a small training size of $80$ and progressively increase it to $240$. The policy is computed independently for each training size. We use the same validation instances to compare the policies. The results are reported in Table \ref{tab:ltm_20termstorm}. Given the complexity and scale of 20TERM instance, attaining optimality in $88.9\%$ of the cases is good enough. A similar pattern is observed for STORM, where all feasible instances are optimal, and the maximum absolute feasibility gap is low. More than $99\%$ feasible training instances with training size $240$ indicate that the relaxed form of the training problem \eqref{eq:training_problem2slp} is reliable for unstable and large-scale problems.

\begin{table}[!tb]
\setlength{\tabcolsep}{6pt}
\resizebox{\textwidth}{!}{%
\centering
\begin{tabular}{c!{\vrule width \arrayrulewidth}c c c c c c c}
\hline
Rep & Size & \# of cells & \# of active & Feasible & Optimal & Feasibility gap & Optimality gap \\
 &  & & cuts & instances & instances &  (Maximum) & (Maximum) \\
\hline
\multirow{-1}{*}{1}
 & $200$ & $8$ & $1387$ & $100\%$ & $100\%$ & $10^{-6}$ & 0.00037 \\
\cellcolor{gray!20} & \cellcolor{gray!20}$200$ & \cellcolor{gray!20}\textendash & \cellcolor{gray!20} \textendash & \cellcolor{gray!20}$90.7\%$ & \cellcolor{gray!20}$99.6\%$ & \cellcolor{gray!20}$0.54$ & \cellcolor{gray!20}$0.18$ \\
\hline
\multirow{-1}{*}{2}
& $200$ & $8$ & $1327$ & $100\%$ & $100\%$ & $10^{-6}$ & $0.00033$ \\
\cellcolor{gray!20} & \cellcolor{gray!20}$200$ & \cellcolor{gray!20}\textendash & \cellcolor{gray!20}\textendash & \cellcolor{gray!20}$91.5\%$ & \cellcolor{gray!20}$100\%$ & \cellcolor{gray!20}$1.27$ & \cellcolor{gray!20}$0.00075$ \\
\hline
\multirow{-1}{*}{3}
& $200$ & $7$ & $1274$ & $100\%$ & $100\%$ & $10^{-6}$ & $0.00067$ \\
\cellcolor{gray!20} & \cellcolor{gray!20}$200$ & \cellcolor{gray!20}\textendash & \cellcolor{gray!20}\textendash & \cellcolor{gray!20}$88.3\%$ & \cellcolor{gray!20}$100\%$ & \cellcolor{gray!20}$0.27$ & \cellcolor{gray!20}$0.00038$ \\
\hline
\multirow{-1}{*}{4}
& $200$ & $6$ & $1326$ & $100\%$ & $100\%$ & $10^{-6}$ & $0.00052$ \\
\cellcolor{gray!20} & \cellcolor{gray!20}$200$ & \cellcolor{gray!20}\textendash & \cellcolor{gray!20}\textendash & \cellcolor{gray!20}$98.9\%$ & \cellcolor{gray!20}$100\%$ & \cellcolor{gray!20}$0.62$ & \cellcolor{gray!20}$0.00046$ \\
\hline
\end{tabular}%
}
\caption{Instance name: 4NODE, Right-hand sides generated with linear time series model. Feasibility tolerance: $10^{-6}$, Optimality tolerance: $0.01$. }
\label{tab:ltm_4node}
\vspace*{-10pt}
\end{table}

\begin{table}[!tb]
\setlength{\tabcolsep}{7pt}
\resizebox{\textwidth}{!}{%
\centering
\begin{tabular}{c!{\vrule width \arrayrulewidth}c c c c c c c}
\hline
Instance & Size & \# of cells & \# of active & Feasible & Optimal & Feasibility gap & Optimality gap \\
 Name &  & & cuts & instances & instances &  (Maximum) & (Maximum) \\
\hline
\multirow{-1}{*}{20TERM}
 & $80$ & $29$ & $3276$ & $96.3\%$ & $100\%$ & $9.7$ & $0.00037$ \\
\cellcolor{gray!20} & \cellcolor{gray!20}$60$ & \cellcolor{gray!20}\textendash & \cellcolor{gray!20} \textendash & \cellcolor{gray!20}$43.3\%$ & \cellcolor{gray!20}$61.5\%$ & \cellcolor{gray!20}$22.84$ & \cellcolor{gray!20}$0.044$ \\
\hline
& $160$ & $52$ & $6468$ & $88.1\%$ & $99.3\%$ & $6.14$ & $0.012$ \\
\cellcolor{gray!20} & \cellcolor{gray!20}$60$ & \cellcolor{gray!20}\textendash & \cellcolor{gray!20} \textendash & \cellcolor{gray!20}$63.3\%$ & \cellcolor{gray!20}$94.7\%$ & \cellcolor{gray!20}$52.28$ & \cellcolor{gray!20}$0.053$ \\
\hline
& $240$ & $49$ & $9114$ & $99.2\%$ & $100\%$ & $1.2 \times 10^{-6}$ & $0.0004$ \\
\cellcolor{gray!20} & \cellcolor{gray!20}$60$ & \cellcolor{gray!20}\textendash & \cellcolor{gray!20} \textendash & \cellcolor{gray!20}$75\%$ & \cellcolor{gray!20}$88.9\%$ & \cellcolor{gray!20}$104.3$ & \cellcolor{gray!20}$0.017$ \\
\hline
\multirow{-1}{*}{STORM}
& $80$ & $38$ & $304$ & $100\%$ & $100\%$ & $10^{-6}$ & $0.0005$ \\
\cellcolor{gray!20} & \cellcolor{gray!20}$60$ & \cellcolor{gray!20}\textendash & \cellcolor{gray!20}\textendash & \cellcolor{gray!20}$25\%$ & \cellcolor{gray!20}$100\%$ & \cellcolor{gray!20}$0.21$ & \cellcolor{gray!20}$0.0016$ \\
\hline
& $160$ & $42$ & $465$ & $99.4\%$ & $100\%$ & $0.0018$ & $0.0006$ \\
\cellcolor{gray!20} & \cellcolor{gray!20}$60$ & \cellcolor{gray!20}\textendash & \cellcolor{gray!20}\textendash & \cellcolor{gray!20}$60\%$ & \cellcolor{gray!20}$100\%$ & \cellcolor{gray!20}$0.29$ & \cellcolor{gray!20}$0.0048$ \\
\hline
& $240$ & $57$ & $657$ & $100\%$ & $100\%$ & $10^{-6}$ & $0.0006$ \\
\cellcolor{gray!20} & \cellcolor{gray!20}$60$ & \cellcolor{gray!20}\textendash & \cellcolor{gray!20}\textendash & \cellcolor{gray!20}$63.3\%$ & \cellcolor{gray!20}$100\%$ & \cellcolor{gray!20}$0.7$ & \cellcolor{gray!20}$0.0019$ \\
\hline
\end{tabular}%
}
\caption{Right-hand sides generated with linear time series model. Feasibility tolerance: $10^{-6}$, Optimality tolerance: $0.01$.}
\label{tab:ltm_20termstorm}
\vspace*{-8pt}
\end{table}

\subsubsection{Sequential Procedure with L-Shaped} \label{sec:lshapedseqproc}

We analyze the policy obtained from the sequential procedure in Algorithm \ref{alg:seqproc_lshaped} for PGP2 and CEP. As noted in the previous section, the policy performs well in these instances when the right-hand sides are generated from a linear time-series model. We thus proceed with LHC sampling to generate right-hand sides, as this technique reveals intriguing properties of the sequential procedure. 

We generate a pool of $5,000$ right-hand sides using LHC sampling. At every round of the procedure, $n_t$ number of right-hand sides are drawn randomly with replacement from this pool. We start with an initial number of right-hand sides, $n_0 = 2$, and increase it by a factor of $1.1$ at every round. We set $z_\mu = 1.96$ and use a tolerance of $10^{-4}$ for the confidence interval widths to ensure feasibility and optimality. We run the procedure for at least $20$ rounds and terminate after $80$ rounds. At termination, the policy is tested on a new set of $1000$ instances. We report the results in Table \ref{tab:seqproc_pgp2cep}. The results indicate that the policy performs well on unseen problems, with CI widths within $0.6\%$. The training problems utilize a significantly smaller number of data points than the total number of observations encountered in the procedure. For example, the final training problems for PGP2 and CEP use $250$ and $298$ instances, respectively, out of a total of $n_T \approx 15,000$ instances. This level of reduction is achieved because we append only the instances rendered infeasible or suboptimal by the policy available at the end of a given round. By doing so, we feed only new bases into the training problem while retaining the bases obtained in previous executions.

We examine the impact on the number of cells, the number of active cuts, the training data size, and the distribution of right-hand sides across cells as the procedure progresses. We present the outcome by plotting each of these quantities against the number of rounds, and it is depicted in Figure \ref{figpgp2:seqproc} for PGP2. The plots show that the fractions of feasible and optimal instances converge to $1$ after around $40$ rounds (see Figures \ref{figpgp2:feas}-\ref{figpgp2:opt}). It indicates that the policy has recovered most of the optimal pieces over the support set $\set{B}$. This behavior is further corroborated by the saturation of the number of cells after $60$ rounds, as illustrated in Figure \ref{figpgp2:cells}.
The size of the training set is also stabilized after $60$ rounds, as illustrated in Figure \ref{figpgp2:cells}. We also present the proportion of right-hand sides in each cell (see Figure \ref{figpgp2:celldist}), which shows that the second cell has the highest number of right-hand sides and most cells have less than $1\%$ of the right-hand sides. Similar results are observed for the CEP instance and are presented in Figure \ref{figcep:seqproc} in Appendix \ref{app:seqproccep}.

\begin{table}[!tb]
\scriptsize
\centering
\resizebox{\textwidth}{!}{%
\begin{tabular}{c!{\vrule width \arrayrulewidth}c c c c c c c}
\hline
Instance & Size & \# of cells & \# of active & Feasible & Optimal & Feasibility gap & Optimality gap \\
 Name &  & & cuts & instances & instances &  (Maximum) & (Maximum) \\
\hline
\multirow{-1}{*}{PGP2}
 & $250$ & $101$ & $505$ & $100\%$ & $100\%$ & $10^{-6}$ & $0.0008$ \\
\cellcolor{gray!20} & \cellcolor{gray!20}$1000$ & \cellcolor{gray!20}\textendash & \cellcolor{gray!20} \textendash & \cellcolor{gray!20}$99\%$ $(\pm 0.6\%)$ & \cellcolor{gray!20}$99.6\%$ $(\pm 0.4\%)$ & \cellcolor{gray!20}$2.6$ & \cellcolor{gray!20}$0.022$ \\
\hline
\multirow{-1}{*}{CEP}
& $298$ & $40$ & $127$ & $100\%$ & $100\%$ & $10^{-6}$ & $0.00032$ \\
\cellcolor{gray!20} & \cellcolor{gray!20}$1000$ & \cellcolor{gray!20}\textendash & \cellcolor{gray!20}\textendash & \cellcolor{gray!20}$99.3\%$ $(\pm 0.5\%)$ & \cellcolor{gray!20}$99.8\%$ $(\pm 0.3\%)$ & \cellcolor{gray!20}$55.42$ & \cellcolor{gray!20}$0.0028$ \\
\hline
\end{tabular}%
}
\caption{Right-hand sides generated with LHC sampling. Feasibility tolerance: $10^{-6}$, Optimality tolerance: $0.001$.}
\label{tab:seqproc_pgp2cep}
\end{table}
\begin{figure}[!tb]
    \centering
    \begin{subfigure}[b]{0.25\textwidth}
        \includegraphics[width=\textwidth]{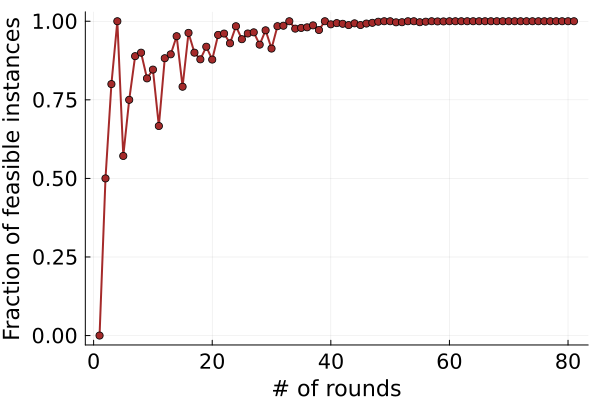}
        \caption{}
        \label{figpgp2:feas}
    \end{subfigure}
    \hfill
    \begin{subfigure}[b]{0.25\textwidth}
        \includegraphics[width=\textwidth]{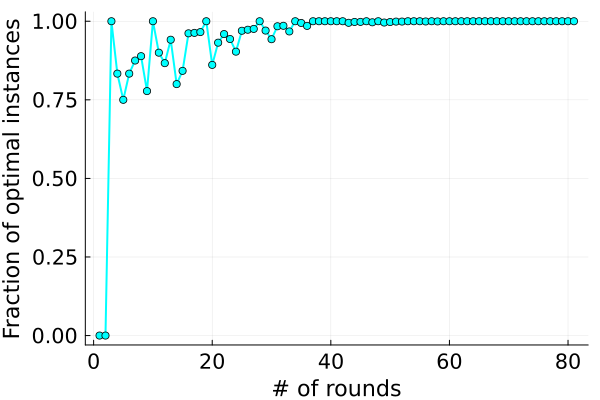}
        \caption{}
        \label{figpgp2:opt}
    \end{subfigure}
        \hfill
    \begin{subfigure}[b]{0.25\textwidth}
        \includegraphics[width=\textwidth]{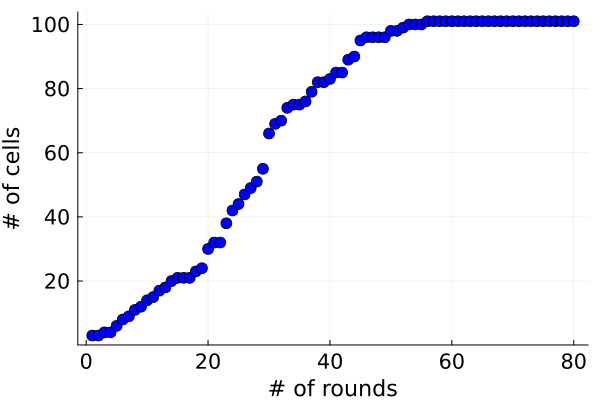}
        \caption{}
        \label{figpgp2:cells}
    \end{subfigure}
    \begin{subfigure}[b]{0.25\textwidth}
        \includegraphics[width=\textwidth]{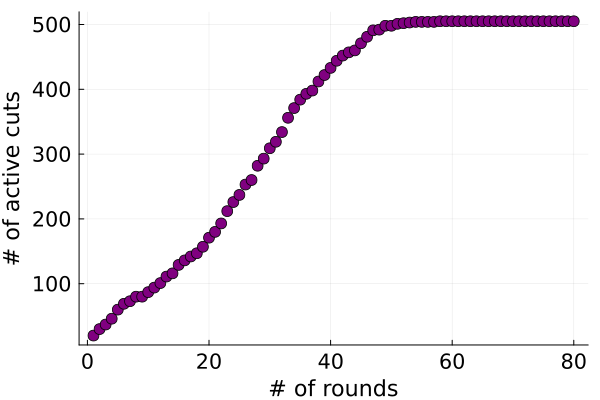}
        \caption{}
        \label{figpgp2:cuts}
    \end{subfigure}
    \hfill
    \begin{subfigure}[b]{0.25\textwidth}
        \includegraphics[width=\textwidth]{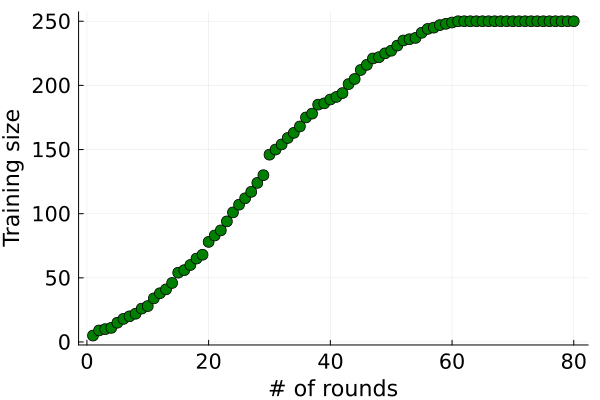}
        \caption{}
        \label{figpgp2:trainsize}
    \end{subfigure}
    \hfill
    \begin{subfigure}[b]{0.25\textwidth}
        \includegraphics[width=\textwidth]{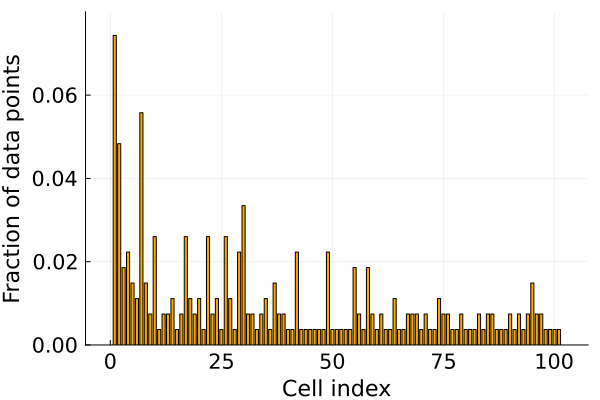}
        \caption{}
        \label{figpgp2:celldist}
    \end{subfigure}
    \caption{(Color online) Stage decomposition method: L-Shaped, Instance name: PGP2, Feasibility tolerance: $10^{-6}$, Optimality tolerance: $0.001$, CI tolerance: $10^{-4}$.}
    \label{figpgp2:seqproc}
    \vspace*{-10pt}
\end{figure}

\subsubsection{Performance of Stochastic Decomposition-Guided Policy}
In this section, we assess the numerical performance of our PLDC policy when the 2-SLPs are solved using the SD method. We use the same four instances, PGP2, CEP, 20TERM, and STORM, described in Table \ref{table:instancesSummary} of the appendix. The right-hand sides for PGP2 and CEP are generated using LHC sampling, while those for 20TERM and STORM are obtained from a linear time-series model, as detailed earlier in this section. The support for second-stage uncertainty is prohibitively large in 20TERM and STORM to recover the exact policy. Therefore, the training problem \eqref{eq:training_problem2slp} can be infeasible. Thus, we relax the constraints of \eqref{eq:training_problem2slp} by adding slack variables to the constraints and minimizing their absolute values. The number of samples used in out-of-sample approximation is set to $2000$ for all the instances. The sample size $M$ for testing the null hypothesis is chosen so that the sample variance of the difference values is below $1\%$.

\textbf{Static Procedure} We solve all the instances for $500$ distinct right-hand sides, out of which we allocate $400$ for the training and $100$ for the policy validation. Table \ref{tab:lhcltm_pgp2cep_sd} demonstrates that all the training instances of PGP2 and CEP are feasible at the policy-prescribed solutions, and the null hypothesis is accepted for all of them. More than $90\%$ of the validation instances are feasible, and the null hypothesis is accepted for more than $80\%$ of the feasible instances. We report the maximum relative mean difference value $\widebar{D}$ across the instances in the last column of the Table \ref{tab:lhcltm_pgp2cep_sd}, which is small for all the instances. This outcome shows that the policy delivers out-of-sample objective value comparable to that at the SD solution. For 20TERM, over $95\%$ of training instances are feasible, with $97\%$ feasibility in validation and policy-prescribed solution acceptance in $96.9\%$ of cases. For STORM, $77\%$ of instances are feasible with a very small maximum feasibility gap, and the policy is acceptable for $98.7\%$ of these.
\begin{table}[!t]
\setlength{\tabcolsep}{6pt}
\resizebox{\textwidth}{!}{%
\centering
\begin{tabular}{c!{\vrule width \arrayrulewidth}c c c c c c c}
\hline
Instance & Size & \# of cells & \# of cuts & Feasible & Null hypothesis & Feasibility gap & Relative mean value \\
 Name &  & &  & instances & accepted &  (Maximum) & difference (Maximum) \\
\hline
\multirow{-1}{*}{PGP2}
 & $400$ & $36$ & $82$ & $100\%$ & $100\%$ & $10^{-6}$ & $10^{-14}$ \\
\cellcolor{gray!20} & \cellcolor{gray!20}$100$ & \cellcolor{gray!20}\textendash & \cellcolor{gray!20} \textendash & \cellcolor{gray!20}$94\%$ & \cellcolor{gray!20}$93.6\%$ & \cellcolor{gray!20}$6.33$ & \cellcolor{gray!20}$0.079$ \\
\hline
\multirow{-1}{*}{CEP}
& $400$ & $34$ & $40$ & $100\%$ & $100\%$ & $10^{-6}$ & $10^{-13}$ \\
\cellcolor{gray!20} & \cellcolor{gray!20}$100$ & \cellcolor{gray!20}\textendash & \cellcolor{gray!20}\textendash & \cellcolor{gray!20}$93\%$ & \cellcolor{gray!20}$83.9\%$ & \cellcolor{gray!20}$67.23$ & \cellcolor{gray!20}$0.023$ \\
\hline
\multirow{-1}{*}{20TERM}
& $400$ & $3$ & $400$ & $97\%$ & $97.7\%$ & $0.87$ & $0.04$ \\
\cellcolor{gray!20} & \cellcolor{gray!20}$100$ & \cellcolor{gray!20}\textendash & \cellcolor{gray!20}\textendash & \cellcolor{gray!20}$97\%$ & \cellcolor{gray!20}$96.9\%$ & \cellcolor{gray!20}$0.80$ & \cellcolor{gray!20}$0.04$ \\
\hline
\multirow{-1}{*}{STORM}
& $400$ & $35$ & $202$ & $95\%$ & $99.7\%$ & $0.012$ & $0.0005$ \\
\cellcolor{gray!20} & \cellcolor{gray!20}$100$ & \cellcolor{gray!20}\textendash & \cellcolor{gray!20}\textendash & \cellcolor{gray!20}$77\%$ & \cellcolor{gray!20}$98.7\%$ & \cellcolor{gray!20}$0.1$ & \cellcolor{gray!20}$0.0003$ \\
\hline
\end{tabular}%
}
\caption{2-SLPs are solved with SD, Feasibility tolerance: $10^{-6}$, Null hypothesis: Mean value difference is $10^{-8}$.}
\label{tab:lhcltm_pgp2cep_sd}
\vspace*{-10pt}
\end{table}

\textbf{Sequential Procedure} Using the experiment set up described in section \ref{sec:lshapedseqproc}, we implement the sequential procedure \ref{alg:seqproc_lshaped} with SD method. Over $98\%$ of the validation instances are feasible at the policy-prescribed solutions with a marginal error of at most $0.5\%$ for both PGP2 and CEP as illustrated in Table \ref{tab:SDseqproc_pgp2cep}. The behavior of key quantities, such as the fraction of feasible and optimal validation instances, the number of cells, the count of cuts in the consolidated master problem, the training size, and the distribution of data points across cells is depicted in Figure \ref{figpgp2SD:seqproc} (also see Figure \ref{figcepSD:seqproc} for CEP in Appendix \ref{app:seqprocSD}). The policy is unstable in the initial rounds, as evident from the oscillating behavior of the fraction of feasible instances in Figure \ref{figpgp2SD:feas}. However, the optimal instance's behavior exhibits fewer oscillations and stabilizes after $20$ rounds because the policy-prescribed solution is acceptable for most of the feasible instances (see Table \ref{tab:SDseqproc_pgp2cep}).
\begin{table}[!htbp]
\resizebox{\textwidth}{!}{%
\begin{tabular}{c!{\vrule width \arrayrulewidth}c c c c c c c}
\hline
Instance & Size & \# of cells & \# of cuts & Feasible & Null hypothesis & Feasibility gap & Relative mean value \\
 Name &  & & & instances & accepted &  (Maximum) & difference (Maximum) \\
\hline
\multirow{-1}{*}{PGP2}
 & $358$ & $74$ & $180$ & $100\%$ & $100\%$ & $10^{-6}$ & $6.8 \times 10^{-9}$ \\
\cellcolor{gray!20} & \cellcolor{gray!20}$1000$ & \cellcolor{gray!20}\textendash & \cellcolor{gray!20} \textendash & \cellcolor{gray!20}$98.2\%$ $(\pm 0.4\%)$ & \cellcolor{gray!20}$98.3\%$ $(\pm 0.4\%)$ & \cellcolor{gray!20}$1.11$ & \cellcolor{gray!20}$0.01$ \\
\hline
\multirow{-1}{*}{CEP}
& $428$ & $38$ & $50$ & $100\%$ & $100\%$ & $10^{-6}$ & $0.00032$ \\
\cellcolor{gray!20} & \cellcolor{gray!20}$1000$ & \cellcolor{gray!20}\textendash & \cellcolor{gray!20}\textendash & \cellcolor{gray!20}$99.4\%$ $(\pm 0.2\%)$ & \cellcolor{gray!20}$99.5\%$ $(\pm 0.2\%)$ & \cellcolor{gray!20}$39.79$ & \cellcolor{gray!20}$0.018$ \\
\hline
\end{tabular}%
}
\caption{2-SLPs are solved with SD, Right-hand sides generated with Latin Hypercube Sampling, Feasibility tolerance: $10^{-6}$, Null hypothesis: Mean value difference is $10^{-8}$.}
\label{tab:SDseqproc_pgp2cep}
\end{table}

\begin{figure}[!tb]
    \centering
    \begin{subfigure}[b]{0.25\textwidth}
        \includegraphics[width=\textwidth]{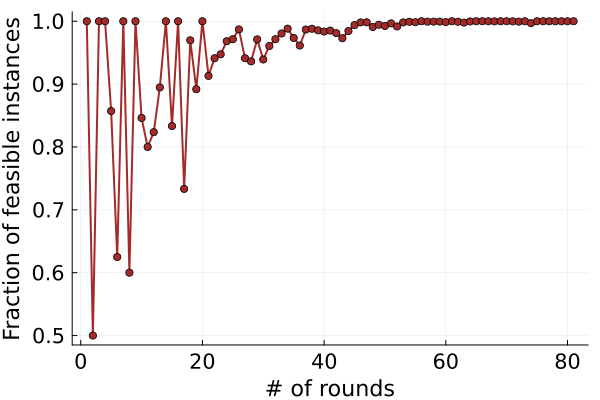}
        \caption{}
        \label{figpgp2SD:feas}
    \end{subfigure}
    \hfill
    \begin{subfigure}[b]{0.25\textwidth}
        \includegraphics[width=\textwidth]{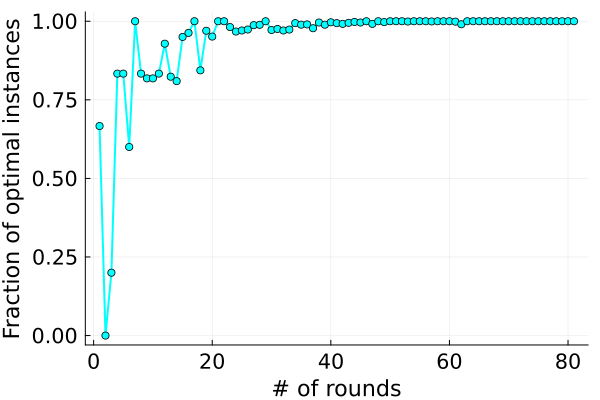}
        \caption{}
        \label{figpgp2SD:opt}
    \end{subfigure}
        \hfill
    \begin{subfigure}[b]{0.25\textwidth}
        \includegraphics[width=\textwidth]{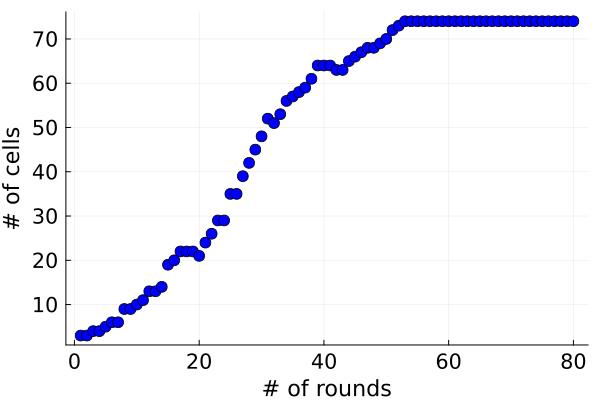}
        \caption{}
        \label{figpgp2SD:cells}
    \end{subfigure}
    \begin{subfigure}[b]{0.25\textwidth}
        \includegraphics[width=\textwidth]{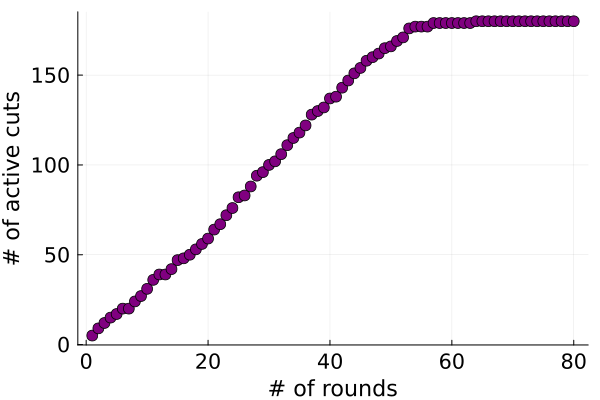}
        \caption{}
        \label{figpgp2SD:cuts}
    \end{subfigure}
    \hfill
    \begin{subfigure}[b]{0.25\textwidth}
        \includegraphics[width=\textwidth]{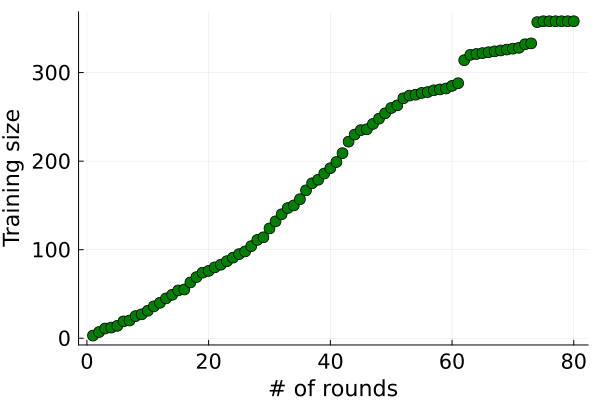}
        \caption{}
        \label{figpgp2SD:trainsize}
    \end{subfigure}
    \hfill
    \begin{subfigure}[b]{0.25\textwidth}
        \includegraphics[width=\textwidth]{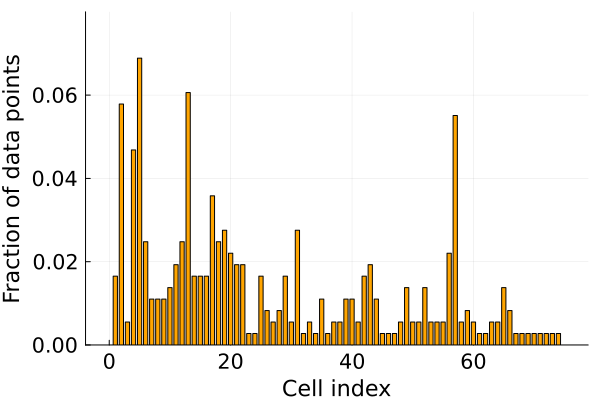}
        \caption{}
        \label{figpgp2SD:celldist}
    \end{subfigure}
    \caption{(Color online) Stage Decomposition Method: SD, Instance name: PGP2, Feasibility tolerance: $10^{-6}$, Relative mean value difference tolerance: $10^{-8}$, CI tolerance: $10^{-4}$.}
    \label{figpgp2SD:seqproc}
\vspace{-8pt}
\end{figure}

\section{Conclusion} \label{sec:conclusion}
In this paper, we developed data-driven PLDC policies to solve 2-SLPs with varying right-hand sides for the first-stage constraints. The policies are trained using information collected from the execution of stage decomposition-based solution algorithms on a set of 2-SLP instances. Specifically, these policies characterize the optimal bases derived from previous solves as vector-valued piecewise-linear functions, with the number of pieces equal to the number of unique optimal bases identified. We proposed a training problem that can be solved using off-the-shelf solvers and whose optimal solution determines the parameters of such piecewise-linear functions.

We presented static procedures for cases in which the training data is collected during execution of the L-shaped and SD methods. These procedures account for differences in solution methods when creating the cells necessary for policy design. We also presented a sequential policy design process that monitors policy quality and allows the addition of new bases information to the policy if needed. The finite-round convergence of this process confirms that our PLDC policy recovers the pieces of the optimal policy over the support set of right-hand-side vectors after solving a finite number of 2-SLP instances. After this round, the policy is guaranteed to deliver a high-quality solution for new instances and, in fact, recovers the optimal solution for right-hand sides that lie within the convex hull of observed cells. The empirical study demonstrates that the policy-prescribed solutions are feasible for a significant number of new instances, and in fact, optimal for these instances. The computational results from the sequential procedure indicate that it efficiently uses the information revealed in new solves to continuously improve the policy's performance. 

It is important to note that the effort required to train the policy is one-time, and repeatedly identifying a high-quality solution from the trained policy requires minimal computational effort, as it only involves evaluating a function with sparse parameters. In contrast, solving 2-SLP problems from scratch every time the parameter values change is computationally expensive. Therefore, utilizing the policy will enable us to prescribe solutions within tight time bounds associated with decision-making settings in practice (recall the ED problem in power systems). 

While the focus of this paper was on the first-stage parameter, a 2-SLP problem is also parametrized by the second-stage uncertainty parameters. Many practical settings that can be modeled as 2-SLP require adapting their solutions to parameter values that drift over time, and it is natural to ask how to design policies in this context. The policy presented in this paper also provides the necessary foundation for extending it to multi-stage SLP, where the right-hand side at each stage is random and depends on the previous-stage decision. We will undertake these intriguing extensions in our future research. 

\section*{Acknowledgment}
We would like to express our sincere gratitude to the U.S Department of Energy-Office of Science, DE-SC0023361, for the financial support to this research. We are also grateful to Southern Methodist University for its support and assistance in providing the necessary computational and working facilities throughout the course of this work.

\bibliographystyle{apalike}
\bibliography{ddpolicies}

\appendix
\section*{Appendix: Omitted Proofs and Numerical Results}

\section{Proof of Lemma \ref{lem:single_subgrad_existence}} \label{appendix:lemmaproof}
\proof{Proof.}
To prove this lemma, it suffices to show the existence of $u^\ell_k$ and $v^\ell_k$ that are valid subgradients of $(\phi_1(\cdot))_k$ and $(\phi_{2}(\cdot))_k$ at all points belonging to the convex hull of $\cC^\ell$. We establish this as follows. Using \cite{kripfganz1987piecewise}, $(\phi_1(b))_k$ can be chosen as a sum of finitely many once-folded PL convex functions. Each once-folded PL convex function has a kink at the intersection of two neighboring cells (see Figure $1$ in \cite{kripfganz1987piecewise}). It implies that $(\phi_1(b))_k$ restricted to any cell is linear. The second function $(\phi_2(b))_k$ can be set to $(\phi_1(b))_k - (\phi(b))_k$. Clearly, $(\phi_2(b))_k$ restricted to any cell is also linear. Therefore, there exists convex functions $(\phi_1(\cdot))_k$ and $(\phi_2(\cdot))_k$ such that restricting them to $\text{Conv}(\cC^\ell)$ are linear functions. Hence, there exists $u^\ell_k \in \partial (\phi_1(b))_k$ and $v^\ell_k \in \partial (\phi_2(b))_k$ for all $b \in \text{Conv}(\cC^\ell)$.\hfill \ifpaper \Halmos \else \qed \fi

\section{Proof of Theorem \ref{thm:finiteroundconv}} \label{app:finiteroundproof}

\proof{Proof.} Consider the collection of $m_1+1$ vectors from a cell $\bar{\cC}^\ell$, denoted as $\bar{\cD}^\ell = \{b^{\ell,i_0},b^{\ell,i_1},\ldots,b^{\ell,i_{m_1}}\}$. By assumption, the vectors $b^{\ell,i_1}-b^{\ell,i_0},\ldots,b^{\ell,i_{m_1}}-b^{\ell,0}$ are linearly independent. The probability of observing each element of $\cB$ is almost surely positive, and consequently, the probability of observing all elements of $\bar{\cD}^\ell$ is one. Furthermore, the number of pieces in the optimal policy is finite. Therefore, there exists a finite round $t^\star$ with probability one, such that the procedure discovers $\bar{L}$ cells, $\cC^1, \cC^2, \ldots, \cC^{\bar{L}}$, and each contains the set $\bar{\cD}^\ell$. It means that, at round $t^\star$, the procedure discovers $\bar{L}$ cells, each containing at least $m_1 +1$ vectors. \par
At round $t$, we have a policy $\phi^t$ and a batch of right-hand side vectors $\widehat{\set{B}}_t$ of size $n_t$. For notational simplicity, we first consider the policy structure described in \eqref{eq:pldc_policy} for LPs, i.e, $\phi^t(b) = x^\star(b)$ and prove the optimal piece recovery using $m_1 +1$ right-hand side vectors. We will exploit it for 2-SLPs, as they have an additional variable $\eta$, that is, $\phi^t(b) = (\hat{x}(b),\hat{\eta}(b))$. \par 

The support set $\set{B}$ is decomposed into $\bar{L}$ cells i.e., $\set{B} = \cup_{\ell=1}^{\bar{L}} \bar{\cC}^\ell$, where $\cC^\ell$ is associated with an optimal basis $\bfB^\ell$. The $\ell$-th piece of the optimal policy is an affine function given by $((\bfB^\ell)^{-1}b)_k + h^{\star,\ell}_k$ for all $k \in B^\ell$ and is a constant zero function for all $k \in N^\ell$. Our first and most crucial step is to prove that there exists a finite round $t^\star$ such that for all $t \geq t^\star$, $(\phi^t(b))_k = ((\bfB^\ell)^{-1}b)_k + h^{\star,\ell}_k$ for all $k \in B^\ell$ and $(\phi^{t}(b))_k = 0$ for all $k \in N^\ell$, for all $\ell \in [\bar{L}]$. This result indicates that the PLDC policy $\phi^t$ has recovered the optimal piece for every cell in $t^\star$ rounds. We prove it next. \par
According to Lemma \ref{lem:single_subgrad_existence}, the policy $\phi^{t^\star}(b)$, when restricted to cell $\cC^\ell$, is an affine function. In other words, for all $b' \in \cC^\ell$, we have
\begin{align*}
(\phi_1(b'))_k & = \inner{u^{\ell}_k, b'-b^\ell} + x^\star(b^\ell)_k + z^\ell_k,  \\
(\phi_2(b'))_k & = \inner{v^{\ell}_k, b'-b^\ell} + z^\ell_k. 
\end{align*}
Using this, we can compute $\phi^{t^\star}(b')$ as follows:
\begin{align}
(\phi^{t^\star}(b'))_k &= (\phi_1(b'))_k -(\phi_2(b'))_k \nonumber \\
& = \inner{u^{\ell}_k - v^{\ell}_k, b'} \nonumber \\
& \ \ \ + \inner{u^{\ell}_k - v^{\ell}_k, -b^\ell} + x^\star(b^\ell)_k \\
& =: \inner{g^\ell_k,b'} + h^\ell_k. \label{eq:phiaffinefunc}
\end{align}
After $t^\star$ rounds, we have $\bar{\cD}^\ell \subseteq \cC^\ell$. Therefore, $\phi^{t^\star}(b^{\ell,i_j}) = x^\star(b^{\ell,i_j})$ for all $j \in \{0\} \cup [m_1]$ from Theorem \ref{thm:pldc_properties}. We also have $\phi^\star(b^{\ell,i_j}) = x^\star(b^{\ell,i_j})$ because $\bar{\cC}^\ell$ is associated with optimal basis $\bfB^\ell$. It implies that $\phi^{t^\star}(b^{\ell,i_j}) - \phi^\star(b^{\ell,i_j}) = 0$ for all $b^{\ell,i_j} \in \bar{\cD}^\ell$. This yields
\begin{align}
\inner{g^\ell_k-g^{\star,\ell}_k,b^{\ell,i_j} } & = h^{\star,\ell}_k - h^\ell_k, \label{eq:policydiffbij}\\
\inner{g^\ell_k-g^{\star,\ell}_k,b^{\ell,i_0}} & = h^{\star,\ell}_k - h^\ell_k. \label{eq:policydiffbi0}
\end{align}
By utilizing \eqref{eq:policydiffbij} and \eqref{eq:policydiffbi0}, we obtain $\inner{g^\ell_k-g^{\star,\ell}_k,b^{\ell,i_j} - b^{\ell,i_0}} = 0$ for all $j \in [m_1]$. This can be expressed in the matrix form as $(g^\ell_k-g^{\star,\ell}_k)^\top Y = \mathbf{0}$, where $Y \in \RR^{m_1 \times m_1}$ is an invertible matrix with its $j$-th column being $b^{\ell,i_j} - b^{\ell,i_0}$. Thus, $(g^\ell_k-g^{\star,\ell}_k)^\top = \mathbf{0}Y^{-1} = \mathbf{0}$, indicating that $g^\ell_k = g^{\star,\ell}_k$. Plugging this into \eqref{eq:policydiffbij}, we get $h^{\star,\ell}_k = h^\ell_k$, which means the offsets of $\phi^{t^\star}$ and $\phi^\star$ are identical. Therefore, at round $t^\star$, the PLDC policy recovers the optimal piece in each cell, meaning for each $\ell$-th cell, 
\begin{align}
(\phi^{t^\star}(b'))_k = ((\bfB^\ell)^{-1}b')_k + (h^{\star,\ell})_k.
\end{align}
Hence, if our procedure observes right-hand side vectors after $t^\star$ rounds, for which $\bfB^\ell$ is an optimal basis, the policy prescribes an optimal solution.\par
\textbf{Proof of part 1.}  The cells for the PLDC policy \eqref{eq:pldc_2slp} are derived from the \eqref{eq:consolmasterproblem} problem. At a given round $t$, the consolidated set $\set{A}$ comprises all pieces of the polyhedral function that were active at the L-Shaped solutions obtained in previous rounds. Since $\cD^\ell \subseteq \cC^\ell$ at round $t^\star$, it follows that $\cA$ includes all pieces that, when incorporated into \eqref{eq:consolmasterproblem}, generate optimal bases for all $b \in \cD^\ell$. Define $\bar{t} := t^\star + 1$. The policy at this round takes the following form for all $b \in \cC^\ell$:
\begin{align}
    \phi^{\bar{t}}(\rhsobs) & = u^\ell \left(\begin{pmatrix} \rhsobs \\ \bfalpha \end{pmatrix} - \begin{pmatrix} b^\ell \\ \bfalpha \end{pmatrix} \right) + \begin{pmatrix} x^\star(\rhsobs^\ell) \\ \eta^\star(b^\ell) \end{pmatrix} + z^\ell  \nonumber \\ 
& \qquad - v^{\ell} \left(\begin{pmatrix} \rhsobs \\ \bfalpha \end{pmatrix} - \begin{pmatrix} b^\ell \\ \bfalpha \end{pmatrix} \right) - z^\ell \nonumber \\
& = (\hat{u}^\ell-\hat{v}^\ell)(b-b^\ell) + \begin{pmatrix} x^\star(\rhsobs^\ell) \\ \eta^\star(b^\ell) \end{pmatrix} \nonumber \\ 
& = (\hat{u}^\ell-\hat{v}^\ell) b - (\hat{u}^\ell-\hat{v}^\ell) b^\ell + \begin{pmatrix} x^\star(\rhsobs^\ell) \\ \eta^\star(b^\ell) \end{pmatrix}, \nonumber\\
& =: \hat{g}^\ell b + \hat{h}^\ell. \label{eq:pldc_2slp_cell}
\end{align}
Here, $\hat{u}^\ell$ and $\hat{v}^\ell$ are the matrices of dimension $(d_x+1) \times m_1$, obtained by extracting first $m_1$ columns of matrices $u^\ell$ and $v^\ell$ respectively. Notice that $\phi^{\bar{t}}(b)$ computes the first-stage decision $x$ and epigraph variable $\eta$. However, $\phi^\star(b)$ yields values of slack variables as well corresponding to inequality constraints in \eqref{eq:consolmasterproblem}. Since, we are interested only in $(x,\eta)$, we focus on the indices of these variables only. From now onward, $B^\ell$ and $N^\ell$ denote the basic and nonbasic variables indices in $(x,\eta)$. The optimal policy $\phi^\star$ over a cell $\cC^\ell$ can be expressed as
\begin{align}
(\phi^\star(b))_k & = ((\bfB^\ell)^{-1}[b;\mathbf{\alpha}])_k + \tilde{h}^{\ell}_k \nonumber \\
& =: g^{\star,\ell}_kb + h^{\star,\ell}_k,  \label{eq:optpolcell}
\end{align}
where $\mathbf{\alpha} \in \RR^{\vert \set{A} \vert}$ and $g^{\star,\ell}_k$ is a row vector of size $m_1$. From Theorem \ref{thm:pldc_properties} and using the fact that $\bar{t} \geq t^\star+1$, we get 
\begin{align}
(\phi^{\bar{t}}(b^{\ell,i_j}))_k - (\phi^\star(b^{\ell,i_j}))_k = 0, \label{eq:pldcpoloptpoldiff}
\end{align}
for all $b^{\ell,i_j} \in \bar{\cD}^\ell, k \in B^\ell \cup N^\ell$. Upon substituting \eqref{eq:pldc_2slp_cell} and \eqref{eq:optpolcell} in \eqref{eq:pldcpoloptpoldiff}, we get
\begin{align}
(\hat{g}^\ell_k - g^{\star,\ell}_k) b^{\ell,i_j}  & =  h^{\star,\ell}_k - \hat{h}^\ell_k, \ \forall \ j \in [m_1], \label{eq:diffblij} \\
(\hat{g}^\ell_k - g^{\star,\ell}_k)b^{\ell,i_0} & = h^{\star,\ell}_k - \hat{h}^\ell_k. \label{eq:diffbli0}
\end{align}
Combining \eqref{eq:diffblij}-\eqref{eq:diffbli0}, we get
\begin{align}
(\hat{g}^\ell_k - g^{\star,\ell}_k) (b^{\ell,i_j} - b^{\ell,i_0}) = 0 \ \forall \ j \in [m_1].
\end{align}
Using previous arguments, row vector $(\hat{g}^\ell_k - g^{\star,\ell}_k)$ must be equal to zero as it is orthogonal to a basis of $m_1$ dimensional space. Therefore, $\hat{g}^\ell_k =g^{\star,\ell}_k$. Substituting it \eqref{eq:diffblij}, we obtain $h^{\star,\ell}_k = \hat{h}^\ell_k$. Hence, after $\bar{t}$ rounds, the PLDC policy recovers the optimal piece of every cell. \par

Therefore, $R_{\set{X}}^{n_{t}}(\phi^{t}) = 0$ and $p_{m_t}(\phi^t) = 0$ almost surely, for all $t \geq \bar{t}$. We have a feasibility stopping time $T_{\set{X}}(\epsilon) = \inf_{t \geq 1} \{ t \ \vert \ R_{\set{X}}^{n_t}(\phi^t) + \frac{z_{\mu/2}}{\sqrt{4n_t}} \leq \epsilon \}$. Then,
\begin{align*}
& \mathbb{P}(T_{\set{X}}(\epsilon) = \infty) \\
& \leq \lim_{t \to \infty} \mathbb{P} \Big( R_{\set{X}}^{n_t}(\phi^t) + \frac{z_{\mu/2}}{\sqrt{4n_t}} > \epsilon \Big) \\
& = 0.
\end{align*}
The last equality holds because $R_{\set{X}}^{n_t}(\phi^t) = 0$ for all $t \geq \bar{t}$ and $\frac{z_{\mu/2}}{\sqrt{4n_t}} \to 0$  as $t \to \infty$. Hence, $\mathbb{P}(T_{\set{X}}(\epsilon) < \infty) = 1$. Similarly, consider an optimality stopping time
\begin{align}
    T_{f}(\epsilon) & = \inf_{t \geq 1} \{ t \ \vert \ p_{m_t}(\phi^t) + \frac{z_{\mu/2}}{\sqrt{4m_t}} \leq \epsilon \}.
\end{align}
Using the fact that $p_{m_t}(\phi^t) = 0$ almost surely for all $t \geq \bar{t}$, we get $\mathbb{P}(T_{f}(\epsilon) < \infty) = 1$. This completes the proof of part $1$. \par
\textbf{Proof of part 2.} Recall that $\widehat{\set{B}}$ is the training data of right-hand side vectors at round $t$ of Algorithm \ref{alg:seqproc_lshaped}. The optimal value of 2-SLP \eqref{eq:2slp} at $b = b^i$, denoted by $f^\star(b^i)$, satisfies $f^\star(b^i) \leq c^\top x^\star(b^i)+ \expect{Q(x^\star(b^i),\rv)}{}$ for all $b^i \in \widehat{\set{B}}$. Since $x^\star(b^i)$ is optimal almost surely, the upper bound is within an $\epsilon$-neighborhood of $f^\star(b^i)$ (see Theorem 2 in \citep{Higle.Sen-1991} and Theorem 5 in \citep{Higle.Sen-1994}). For a given $\epsilon' ( < \epsilon)$ and any $b^i \in \widehat{\set{B}}$, there exists a sample size $N_i$ such that the sample variance of $\{Q(x^\star(b^i),\obs^j)\}_{j\in [N_i]}$ is below a threshold $\epsilon'$. Therefore, at a given round $t$, we can choose $N \geq \max_{i \in \vert \widehat{\set{B}} \vert } \{N_i\}$ to compute out-of-sample approximation \eqref{eq:outofsamplelbs}. Using the sample size $N$, the variance of $\{Q(x^\star(b^i),\obs^j)\}_{j\in [N]}$ is less than a threshold $\epsilon'$ for all $b^i \in \widehat{\set{B}}$. As a result, the out-of-sample function $c^\top x^\star(b^i) + \frac{1}{N}\sum_{j=1}^NQ(x^\star(b^i),\obs^j)$ is a high quality estimate for all $b^i \in \widehat{\set{B}}$, in the sense that out-of-sample evaluation lies in the $\epsilon$-neighborhood of $f^\star(b^i)$. This implies that the affine pieces in $\bar{\set{A}}$ are also of high quality for all right-hand sides belonging to set $\widehat{\set{B}}$. Now using the arguments made earlier, we have $\cD^\ell \subseteq \cC^\ell$ for all $t \geq t^\star$. Therefore, the consolidated problem \eqref{eq:consolmasterproblem} constructed with $\bar{\set{A}}$ yields bases for all $b \in \cD^\ell$ which are optimal with probability one, i.e, $\phi^t(b) = (x^\star(b),\eta^\star(b))$.

By repeating the steps discussed in part 1 to recover the optimal piece, the PLDC policy $\phi^t(b)$ equals the statistically optimal piece in every cell for all $t \geq t^\star+1 =: \bar{t}_{sd}$. It means after a finite number of rounds $\bar{t}_{sd}$, the policy-prescribed solution $\hat{x}^t(b)$ equals the SD solution $x^\star(b)$ for all $b \in \set{B}$. Hence, the function values $c^\top \hat{x}^t(b) + Q(\hat{x}^t(b),\obs^s)$ and $c^\top x^\star(b) + Q(x^\star(b),\obs^s)$ are equal almost surely for all samples $s \in [M]$ used to calculate the mean estimator of gap $f(x^\star(b))-f(\hat{x}^t(b))$. As a result, the null hypothesis is accepted with high probability and hence $r_{m_t}(\phi^t) = 0$ for all $t \geq \bar{t}_{sd}$ with probability one. Thus, $\mathbb{P}(T_{f}(\epsilon) = \infty) = 0$.
\hfill \ifpaper \Halmos \else \qed \fi

\section{Test Instances} \label{sec:testinstances}
We use instances of five problems, denoted PGP2, CEP, 4NODE, 20TERM, and STORM, for our experiments. The first is a power generation problem, while the second addresses a capacity expansion planning problem. Both instances are small-scale in terms of the first-stage variables and constraints, as well as the number of second-stage scenarios. However, they exhibit enough variability in the optimal bases for different right-hand sides of the first-stage constraints. The instance 4NODE is a freight fleet scheduling problem. Due to its large scenario size, we solve its Sample Average Approximation (SAA) using the L-Shaped method. 20TERM is a freight transportation problem, and STORM is an air freight scheduling problem. We solve the SAA using the L-Shaped for 20TERM and STORM because they have an exponential number of scenarios. 20TERM and STORM are challenging problems because of their larger sets of dual vertices of the scenario subproblems, as noted in \cite{Sen.Liu-2016}. This results in numerous, vastly different active cuts across the multiple solves, leading to millions of constraints and variables in the training problem. Thus, we determine the policies for 20TERM and STORM on high-performance computing resources. Table \ref{table:instancesSummary} lists the relevant characteristics of the test instances. We refer the reader to \cite{Higle.Sen-1996} for more details about these problems.
\begin{table}[!h]
\centering
\begin{tabular}{c!{\vrule width \arrayrulewidth}c c c c c c}
\hline
Instance & $1$st stage & $2$nd stage & $\#$ of right-hand side & Solve & \\
name & variables/constraints & scenario size & perturbed & SAA (size) \\
\hline
 PGP2$^*$ & $4/2$ & $576$ & $2$ & \xmark(\textendash) &  \\
 CEP$^*$ & $8/5$ & $216$ & $5$ & \xmark(\textendash) &  \\
 4NODE$^*$ & $52/14$ & $32,768$  & $3$ & \cmark $(1000)$ &  \\
 $\text{20TERM}^{\dagger}$ & $63/3$  & $10^{12}$ & $3$ & \cmark $(500)$ & \\
 $\text{STORM}^{\dagger}$ & $121/185$ & $10^{81}$  & $5$ & \cmark $(500)$ & \\
 \hline
\end{tabular}
\caption{Test instances used in the experiments. $*$: policy computed on ordinary laptop machine, $\dagger$: policy computed on high-performance computing machine.}
\label{table:instancesSummary}
\end{table}

\section{Advantages over standard prediction models} \label{appendix:learningPL}
A piecewise linear model can be trained using the training set $\{b^i,x^\star(b^i)\}_{i=1}^n$. For example, \cite{siahkamari2020piecewise} learns the parameters of the PL function by solving a convex problem with linear constraints. If we find these parameters using our training problem \eqref{eq:training_problem}, we get the policy
\begin{align}
\phi(\rhsobs)_k & = \max_{j \in [n]}(u^j_k)^\top (\rhsobs -b^j) + x^\star(\rhsobs^j)_{k} + z^j_k \nonumber \\ 
& \qquad - \max_{j \in [n]} (v^{j}_k)^\top (\rhsobs -b^j)+ z^j_k, \label{eq:policy_npieces}
\end{align}
which exactly resembles the policy developed in \cite{siahkamari2020piecewise}. However, policy \eqref{eq:policy_npieces} suffers from overfitting issues as demonstrated in the Table \ref{tab:policycomp}. For the mean value problem of pgp2, we determine our PLDC policy \eqref{eq:pldc_policy} and the policy described in \eqref{eq:policy_npieces}. It is evident from Table \ref{tab:policycomp} that the PLDC policy performs well and improves as we increase the training size. However, all the validation instances are infeasible at the solution predicted by the policy \eqref{eq:policy_npieces}. This behavior is permissible because when the number of data points increases, the parameter size escalates and the training problem has a broad flexibility to choose the parameters, i.e., it can choose one piece per data point, making it prone to overfitting. 
Therefore, in contrast to machine learning predictions, our goal is to devise the policies that retain the linear pieces associated with each distinct optimal basis found in previous solves. We utilize information from previous solves of decomposition algorithms to enable the policy for prescribing feasible or near-optimal solutions. Our aim is not merely to learn the parameters of a piecewise linear function, but to store previous solve information in a compact form that can be efficiently accessed for future changes to the right-hand side vectors. 
Neural networks are more sophisticated black-box models widely used for prediction. Likewise, in the policy \eqref{eq:policy_npieces}, the training process using a neural network is uninformed of the large number of constraints and the structural properties of the stochastic programs. 
\begin{table}[!tb]
\centering
\resizebox{\textwidth}{!}{%
\begin{tabular}{c!{\vrule width \arrayrulewidth}c c c c c c}
\hline
Policy & Size & \# of cells & Feasible & Optimal & Feasibility gap & Optimality gap \\
 &  & & instances & instances &  (Maximum) & (Maximum) \\
\hline
\multirow{-1}{*}{This work}
 & $160$ & $8$ & $100\%$ & $100\%$ & $10^{-6}$ & $0.0$ \\
\cellcolor{gray!20} & \cellcolor{gray!20}$40$ & \cellcolor{gray!20}\textendash & \cellcolor{gray!20}$100\%$ & \cellcolor{gray!20}$97.5\%$ & \cellcolor{gray!20}$10^{-6}$ & \cellcolor{gray!20}$0.002$ \\
\hdashline
\multirow{-1}{*}{Eq. \ref{eq:policy_npieces}}
& $160$ & \textendash & $100\%$ & $100\%$ & $10^{-6}$ & $0.0$ \\
\cellcolor{gray!20} & \cellcolor{gray!20}$40$ & \cellcolor{gray!20}\textendash & \cellcolor{gray!20}$0.0\%$ & \cellcolor{gray!20}$0.0\%$ & \cellcolor{gray!20}$1676.48$ & \cellcolor{gray!20} \textendash \\
\hline
\multirow{-1}{*}{This work}
 & $320$ & $7$ & $100\%$ & $100\%$ & $10^{-6}$ & $0.0$ \\
\cellcolor{gray!20} & \cellcolor{gray!20}$80$ & \cellcolor{gray!20}\textendash & \cellcolor{gray!20}$98.75\%$ & \cellcolor{gray!20}$100\%$ & \cellcolor{gray!20}$3.92$ & \cellcolor{gray!20}$0.0$ \\
\hdashline
\multirow{-1}{*}{Eq. \ref{eq:policy_npieces}}
& $320$ & \textendash & $100\%$ & $100\%$ & $10^{-6}$ & $0.0$ \\
\cellcolor{gray!20} & \cellcolor{gray!20}$80$ & \cellcolor{gray!20}\textendash & \cellcolor{gray!20}$0.0\%$ & \cellcolor{gray!20}$0.0\%$ & \cellcolor{gray!20}$758.35$ & \cellcolor{gray!20} \textendash \\
\hline
\multirow{-1}{*}{This work}
 & $480$ & $9$ & $100\%$ & $100\%$ & $10^{-6}$ & $0.0$ \\
\cellcolor{gray!20} & \cellcolor{gray!20}$120$ & \cellcolor{gray!20}\textendash & \cellcolor{gray!20}$99.17\%$ & \cellcolor{gray!20}$100\%$ & \cellcolor{gray!20}$0.63$ & \cellcolor{gray!20}$0.0$ \\
\hdashline
\multirow{-1}{*}{Eq. \ref{eq:policy_npieces}}
& $480$ & \textendash & $100\%$ & $100\%$ & $10^{-6}$ & $0.0$ \\
\cellcolor{gray!20} & \cellcolor{gray!20}$120$ & \cellcolor{gray!20}\textendash & \cellcolor{gray!20}$0.0\%$ & \cellcolor{gray!20}$0.0\%$ & \cellcolor{gray!20}$1548.42$ & \cellcolor{gray!20} \textendash \\
\hline
\end{tabular}%
}
\caption{Right-hand sides generated uniformly at random. Feasibility tolerance: $10^{-6}$, Optimality tolerance: $0.001$.}
\label{tab:policycomp}
\end{table}

\section{Results of Sequential Procedure \ref{alg:seqproc_lshaped} with L-Shaped Method} \label{app:seqproccep}

\begin{figure}[H]
    \centering
    \begin{subfigure}[b]{0.27\textwidth}
        \includegraphics[width=\textwidth]{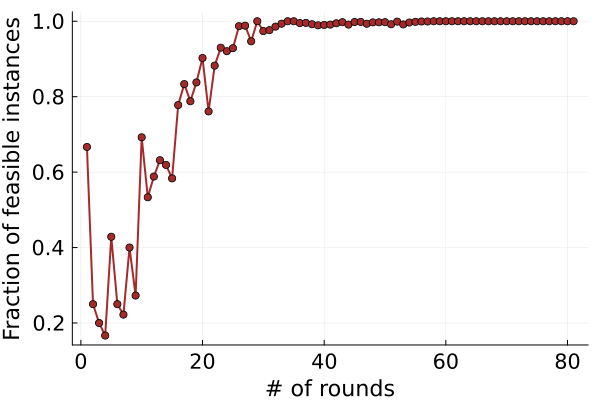}
        \caption{}
        \label{figcep:feas}
    \end{subfigure}
    \hfill
    \begin{subfigure}[b]{0.27\textwidth}
        \includegraphics[width=\textwidth]{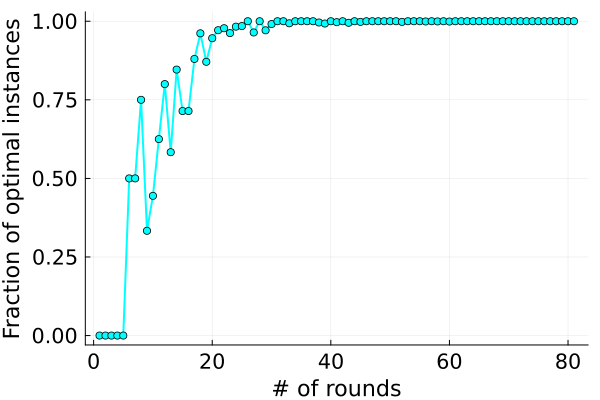}
        \caption{}
        \label{figcep:opt}
    \end{subfigure}
        \hfill
    \begin{subfigure}[b]{0.27\textwidth}
        \includegraphics[width=\textwidth]{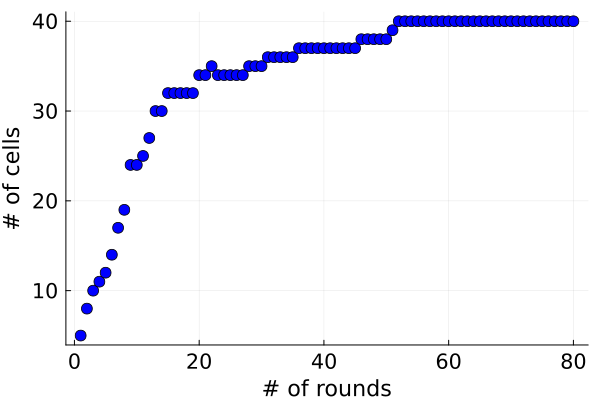}
        \caption{}
        \label{figcep:cells}
    \end{subfigure}

    \vskip\baselineskip

    \begin{subfigure}[b]{0.27\textwidth}
        \includegraphics[width=\textwidth]{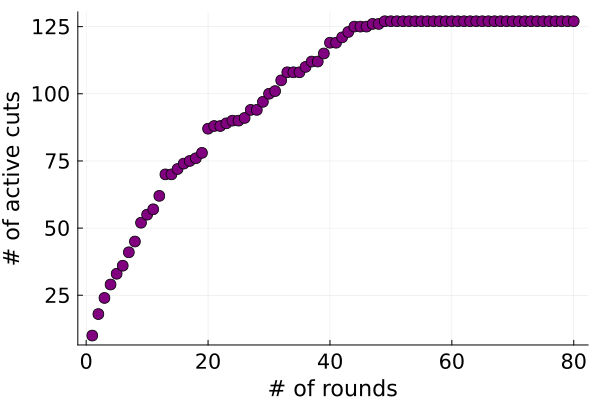}
        \caption{}
        \label{figcep:cuts}
    \end{subfigure}
    \hfill
    \begin{subfigure}[b]{0.27\textwidth}
        \includegraphics[width=\textwidth]{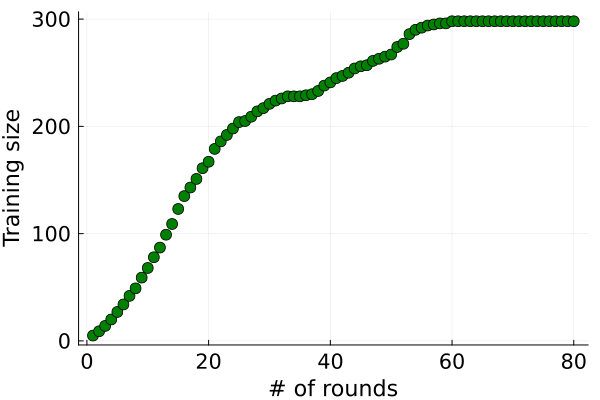}
        \caption{}
        \label{figcep:trainsize}
    \end{subfigure}
    \hfill
    \begin{subfigure}[b]{0.27\textwidth}
        \includegraphics[width=\textwidth]{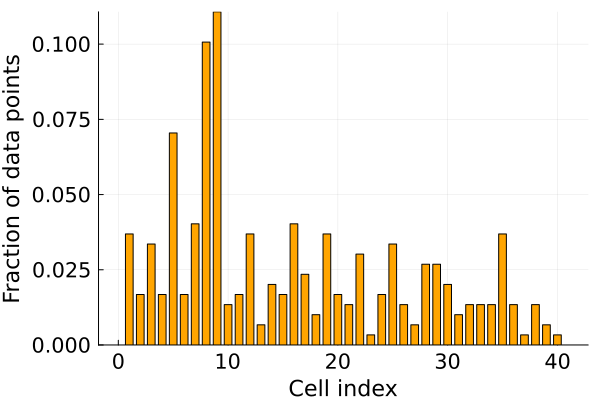}
        \caption{}
        \label{figcep:celldist}
    \end{subfigure}

    \caption{Stage decomposition method: L-Shaped, Instance name: CEP, Feasibility tolerance: $10^{-6}$, Optimality tolerance: $0.001$, CI tolerance: $10^{-4}$.}
    \label{figcep:seqproc}
\end{figure}

\section{Results of Sequential Procedure \ref{alg:seqproc_lshaped} with Stochastic Decomposition Method} \label{app:seqprocSD}

\begin{figure}[H]
    \centering
    \begin{subfigure}[b]{0.27\textwidth}
        \includegraphics[width=\textwidth]{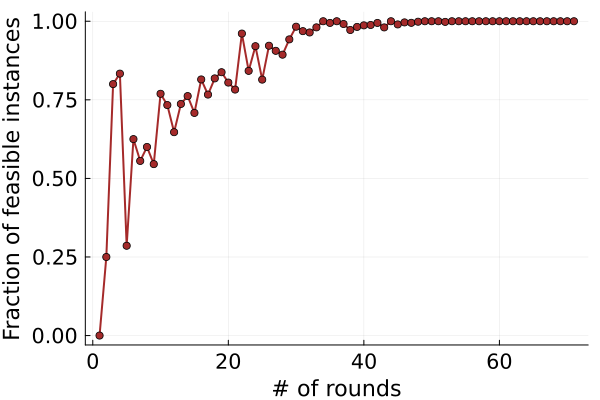}
        \caption{}
        \label{figcep2SD:feas}
    \end{subfigure}
    \hfill
    \begin{subfigure}[b]{0.27\textwidth}
        \includegraphics[width=\textwidth]{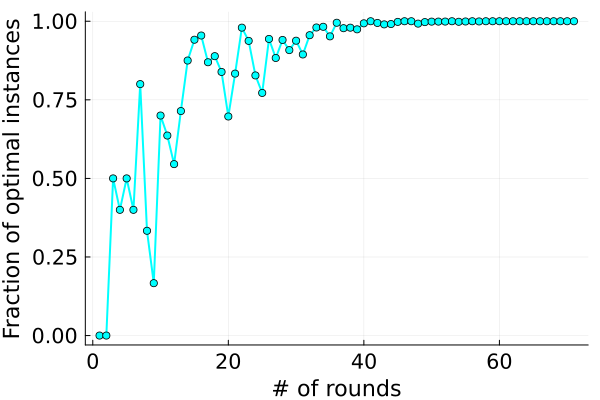}
        \caption{}
        \label{figcepSD:opt}
    \end{subfigure}
        \hfill
    \begin{subfigure}[b]{0.27\textwidth}
        \includegraphics[width=\textwidth]{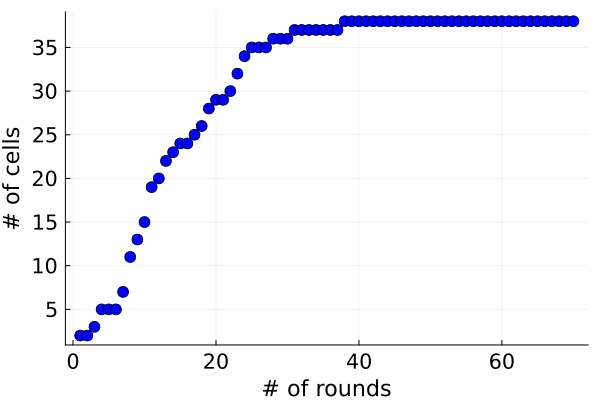}
        \caption{}
        \label{figcepSD:cells}
    \end{subfigure}

    \vskip\baselineskip

    \begin{subfigure}[b]{0.27\textwidth}
        \includegraphics[width=\textwidth]{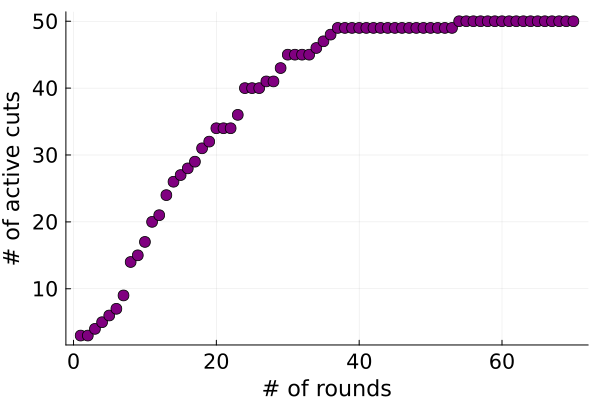}
        \caption{}
        \label{figcep2SD:cuts}
    \end{subfigure}
    \hfill
    \begin{subfigure}[b]{0.27\textwidth}
        \includegraphics[width=\textwidth]{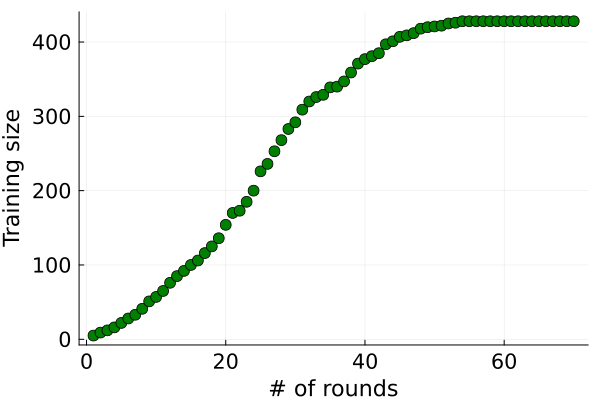}
        \caption{}
        \label{figcep2SD:trainsize}
    \end{subfigure}
    \hfill
    \begin{subfigure}[b]{0.27\textwidth}
        \includegraphics[width=\textwidth]{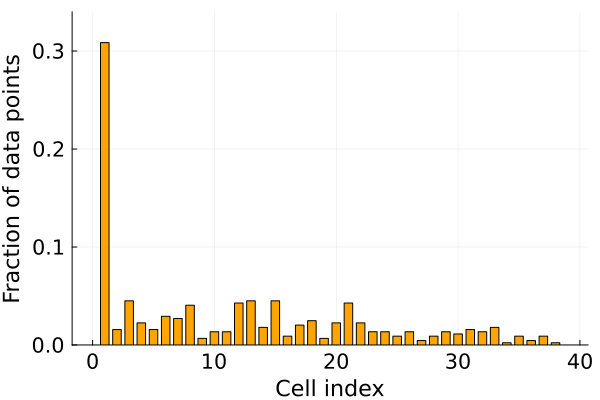}
        \caption{}
        \label{figcep2SD:celldist}
    \end{subfigure}
    \caption{Stage Decomposition Method: SD, Instance name: CEP, Feasibility tolerance: $10^{-6}$, Relative mean value difference tolerance: $10^{-8}$, CI tolerance: $10^{-4}$.}
    \label{figcepSD:seqproc}
\end{figure}


\end{document}